\documentclass[12pt]{article}
\usepackage{amsmath, amsfonts, amssymb, wasysym, graphicx}
\usepackage{hyperref}
\hypersetup{colorlinks,citecolor=blue,linkcolor=blue}

\title{ Relatively Hyperbolic Groups with Semistable Peripheral Subgroups} 

\author{M. Haulmark and M. Mihalik}
\newtheorem{theorem}{Theorem}[section]

\newtheorem{lemma}[theorem]{Lemma}

\newtheorem{conjecture}[theorem]{Conjecture}
\newtheorem{remark}[theorem]{Remark}

\newenvironment{proof}{\addvspace{12pt}\noindent{\bf Proof:}}{
$\Box$\par\addvspace{12pt}}

\newcounter{definitionnum}
\newenvironment{definition}{\addvspace{12pt}\refstepcounter{definitionnum}
\noindent{\bf Definition \arabic{definitionnum}.}}{\par\addvspace{12pt}}

\date{\today}
\begin{document}
\maketitle

\begin{abstract} 
Suppose $G$ is a finitely presented group that is hyperbolic relative to ${\bf P}$ a finite collection of finitely generated proper subgroups of $G$. Our main theorem states that if each $P\in {\bf P}$ has semistable fundamental group at $\infty$, then $G$ has semistable fundamental group at $\infty$. The problem reduces to the case when $G$ and the members of ${\bf P}$ are all one ended and finitely presented. In that case,
if the boundary $\partial (G,{\bf P})$ has no cut point, then $G$ was already known to have semistable fundamental group at $\infty$. We consider the more general situation when $\partial (G,{\bf P})$ contains cut points. \end{abstract}
\section{Introduction}\label{Intro}

We are interested in the asymptotic behavior of relatively hyperbolic groups. We consider a property of finitely presented groups that has been well studied for over 40 years called semistable fundamental group at $\infty$. A locally finite complex $Y$ has semistable fundamental group at $\infty$ if any two proper rays $r,s:[0,\infty)\to Y$ that converge to the same end of $Y$ are properly homotopic in $Y$. A  finitely presented group $G$ has semistable fundamental group at $\infty$ if for some (equivalently any) finite complex $X$ with $\pi_1(X)=G$, the universal cover of $X$ has semistable fundamental group at $\infty$. (See section \ref{SB} for several equivalent notions of semistability.) 
It is unknown at this time, whether or not all finitely presented groups have semistable fundamental group at $\infty$, but in \cite{M87} the problem is reduced to considering 1-ended groups. The finitely presented group $G$ satisfies a weaker geometric condition called semistable first homology at $\infty$ if and only if  $H^2(G:\mathbb ZG)$ is free abelian (see \cite{GM85}). The question of whether or not $H^2(G:\mathbb ZG)$ is free abelian for all finitely presented groups $G$ goes back to H. Hopf. Our main interest is in showing certain relatively hyperbolic groups have semistable fundamental group at $\infty$. The work of B. Bowditch \cite{Bow99B} and G. Swarup \cite{Swarup} shows that if $G$ is a 1-ended word hyperbolic group then $\partial G$, the Gromov boundary of $G$, has no (global) cut point. M. Bestvina and G. Mess \cite{BM91} (Propositions 3.2 and 3.3) show the absence of cut points in $\partial G$ implies $\partial G$ is locally connected. It was pointed out by R. Geoghegan that $G$ has semistable fundamental group at $\infty$ if and only if $\partial G$ has the shape of a locally connected continuum (see \cite{DydakS} for a proof of this fact). In particular, all 1-ended word hyperbolic groups have semistable fundamental group at $\infty$. 

Relatively hyperbolic groups are a much studied generalization of hyperbolic groups. Semistability only makes sense for finitely generated groups. 
We only consider finitely presented groups $G$ in our main result. Later in this section and again in \S \ref{GM}, we say what it means for a finitely generated group to be hyperbolic relative to a finite collection of finitely generated subgroups.  If a finitely generated group $G$ is hyperbolic relative to a collection of finitely generated subgroups ${\bf P}$ the pair $(G,{\bf P})$ has a well-defined compact metric boundary (see \S \ref{GM}), denoted $\partial (G,{\bf P})$.  While all 1-ended hyperbolic groups have locally connected boundary without cut points,  the space $\partial (G,{\bf P})$ may contain cut points.   When $\partial (G,{\bf P})$ is connected, it is locally connected (see Theorem \ref{c=lc}) and the Hahn-Mazurkiewicz Theorem (see Theorem 31.5 of \cite{W70}) implies it is the continuous image of the interval $[0,1]$. This implies  $\partial (G,{\bf P})$
is path connected and locally path connected, these facts are important in our approach to showing the relatively hyperbolic groups we consider have semistable fundamental group at $\infty$.  
The main theorem of \cite{MS18} solves a semistability problem when $\partial (G,{\bf P})$ has no cut point. Note that there is no semistability hypothesis on the peripheral subgroups. 

\begin{theorem} \label{MainHM} (Theorem 1.1, \cite{MS18}) 
Suppose $G$ is a 1-ended finitely generated group that is hyperbolic relative to a collection of 1-ended finitely generated proper subgroups ${\bf P}=\{P_1,\ldots, P_n\}$. If $\partial (G,{\bf P})$ has no cut point, then $G$ has semistable fundamental group at $\infty$.
\end{theorem}

The primary semistability question for relatively hyperbolic groups following Theorem \ref{MainHM} is: 

\begin{conjecture} \label{Conj1} (Conjecture 2.1, \cite{MS18}) 
Suppose $G$ is a finitely generated group that is hyperbolic relative to a finite collection $\{P_1,\ldots, P_n\}$ of proper finitely generated subgroups. If each $P_i$ has semistable fundamental group at $\infty$, then $G$ has semistable fundamental group at $\infty$. 
\end{conjecture}

Some support for this conjecture appeared in the form of a result of C. Hruska and K. Ruane: 

\begin{theorem}\label{HR2} 
(\cite{HR19}, Theorem 1.1) Let $(G,{\bf P})$ be relatively hyperbolic with no non-central element of order two. Assume each peripheral subgroup $P\in {\bf P}$ is slender and coherent and all subgroups of $P$ have semistable fundamental group at $\infty$. Then $G$ has semistable fundamental group at $\infty$.
\end{theorem}

When $G$ is finitely presented, the homology version of the conjecture is resolved by the main theorem of \cite{MS19}. 

\begin{theorem} \label{MainH} (Theorem 1.1, \cite{MS19}) 
Suppose $G$ is a finitely presented group that is hyperbolic relative to a collection of  finitely presented subgroups ${\bf P}=\{P_1,\ldots, P_n\}$. If each group $H^2(P_i, \mathbb ZP_i)$ is free abelian then $H^2(G,\mathbb ZG)$ is free abelian.
\end{theorem}
While several results in \cite{MS19} are useful to us, the techniques of that paper are insufficient to resolve the conjecture. A new idea of nearly geodesic homotopies in a cusped space is developed here and it is fundamental in proving our results.
The main theorem of our paper resolves the conjecture when $G$ is finitely presented.
\begin{theorem} \label{Main} ${\bf (Main)}$ 
Suppose $G$ is a  finitely presented group that is hyperbolic relative to a collection of finitely generated subgroups ${\bf P}=\{P_1,\ldots, P_n\}$. If each $P_i$ has semistable fundamental group at $\infty$ then $G$ has semistable fundamental group at $\infty$.
\end{theorem}

All of our work is done in a ``cusped" space $X$ for $(G,{\bf P})$ (see \S \ref{GM}). When $X$ is Gromov hyperbolic then the pair $(G,{\bf P})$ is said to be {\it relatively hyperbolic} or that $G$ is {\it hyperbolic relative to ${\bf P}$}. This cusped space is  a locally finite 2-complex on which $G$ acts by isometries, but not co-compactly (see \S \ref{GM}). It follows from (\cite{Bow12}, \S 6 and \S 9) that the Bowditch boundary for a relatively hyperbolic pair $(G,{\bf P})$ agrees with the Gromov boundary of $X$. Throughout the paper this boundary is denoted $\partial (G,{\bf P})$ and is called the {\it boundary of the relatively hyperbolic pair $(G,{\bf P})$.} 

The base space $Y$ in $X$ is a universal cover of a finite complex with fundamental group $G$. There are  closed neighborhoods $X_m$ of $Y$ in $X$ which are also universal covers of finite complexes with fundamental group $G$ so $G$ has semistable fundamental group at $\infty$ if and only if some (equivalently any) $X_m$ has semistable fundamental group at $\infty$.  The proof of Theorem \ref{Main1} shows that for some $m$, the space $X_m$ (and hence $G$) has semistable fundamental group at $\infty$. 

Any proper ray in $X_m$ is properly homotopic to a proper ray in $Y$. We show two nearly geodesic rays in $Y$ are properly homotopic in $X$ by nearly geodesic homotopies. Using Theorem \ref{excise}, we cut out disks in the domains of our homotopies on which these homotopies stray out of $X_m$ (for some large fixed integer $m$). The geodesic nature of our homotopies allows us to show that the disks can only occur in a locally finite way (see Claims 1 and 2 of the proof of Theorem \ref{periEdge}) and hence we can properly fill in our homotopies on these disks by homotopies with image in $X_m$. This is where we use the hypothesis that the peripheral subgroups are 1-ended and semistable. The resulting homotopies are then combined in a standard way to finish the proof of the theorem. Our nearly geodesic homotopies and the local finiteness arguments of the Claims are the key insights that drive our proofs. 

The remainder of the paper is organized as follows. The first order of business is to reduce our problem to the case where $G$ is 1-ended and the peripheral subgroups are 1-ended and have semistable fundamental group at $\infty$. This is accomplished in Section \ref{Red}. Once the reduction is accomplished we need to know boundaries of the resulting relatively hyperbolic groups are path connected and locally path connected. This is accomplished in Section \ref{Red}. Finally, groups covered by our Main Theorem and not covered by earlier results are described is Section \ref{Red}.  We develop basic semistability background in Section \ref{SB}. Section \ref{Hyp} is a short section on hyperbolicity and thin triangles.  In Section \ref{GM} we review the construction of cusped spaces for a relatively hyperbolic group and discuss some of its properties. We develop the idea of a {\it filter} and a {\it filter map} in Section \ref{FFM}. Filters are graphs in $[0,1]\times [0,\infty)$ that are geodesically mapped into our cusped space and  allow us to produce nearly geodesic homotopies (filter maps). Theorem \ref{filter} is the main result of this section and all of our initial homotopies are built using this result. Triangulations  of our cusped space and $[0,1]\times [0,\infty)$ are developed in Section \ref{TSH}.  Filters maps are turned into our first simplicial homotopies in this section. Section \ref{prelim} contains several tracking results. For each vertex $v$ of $Y$  we  construction a  geodesic ray $r_v$ in $X$ that stays close to $Y$. If $s_v$ is a geodesic ray at the base point $\ast$ in $Y$ converging to the same boundary point as does $r_v$, then we show that each point of $r_v$ is within $\delta$ (the  hyperbolicity constant for $X$) of $s_v$. The rays $r_v$ are important in the construction of filters and filter maps. Theorem 6.1 of \cite{MS19} is introduced in order to cut out the parts of our simplicial homotopies that leave $X_m$. 
Finally, our Main Theorem is proved in Section \ref{MThm}. 

\section{A Reduction to the One-Ended Case}\label{Red}
We begin with a finitely presented group $G$ and a finite collection ${\bf P}$ of finitely generated subgroups of $G$ such that $G$ is hyperbolic relative to ${\bf P}$. The members of ${\bf P}$ are finitely presented by the following result (proved in \cite {DG13}). For a more general result see \cite{DGO17}, Theorem 2.11.

\begin{theorem} \label{fg=fp} (\cite{DG13})
If the group $G$ is finitely presented and hyperbolic relative to a finite collection of proper finitely generated subgroups $P_i$, then the $P_i$ are finitely presented as well.
\end{theorem} 
The reduction we want comes directly from:

\begin{theorem}\label{Reduce} (Theorem 2.9, \cite{MS18}) 
 If Conjecture \ref{Conj1} holds true for the case when $G$ and each $P_i$ is finitely presented and 1-ended (and each $P_i$ has semistable fundamental group at $\infty$), then the conjecture holds true in the more general setting where $G$ and each $P_i$ is finitely presented (with possibly more than 1-end), as long as the $P_i$ have semistable fundamental group at $\infty$. 
\end{theorem}

In his thesis \cite{Da20} A. Dasgupta proves that the only cut points in a connected boundary of a finitely generated relatively hyperbolic group are parabolic. Dasgupta  combines this result with a result of Bowditch to  prove:

\begin{theorem} \label{c=lc}(\cite{Da20}) When the Bowditch boundary of a finitely generated relatively hyperbolic group is connected, then it is  locally connected.
\end{theorem} 

As noted in the introduction, the Hahn-Mazurkiewicz Theorem combines with Theorem \ref{c=lc} to show:
\begin{theorem}\label{lpc}
If $G$ is finitely generated, 1-ended and hyperbolic relative to a finite collection ${\bf P}$ of finitely generated subgroups then $\partial(G,{\bf P})$ is path connected and locally path connected.
\end{theorem}

Results of B. Bowditch (see Theorem 2.13 of \cite{MS18}) determine that cut points appear in $\partial (G,{\bf P})$ precisely when  $(G,{\bf P})$ admits a non-trivial graph of groups decomposition that is a `proper peripheral splitting'. Notice that in the following combination result of M. Mihalik and S. Tschantz, there is no restriction on the number of ends of any of the groups involved.

\begin{theorem} \cite{MT1992} \label{split} 
Suppose $\mathcal G$ is a finite graph of groups decomposition of the finitely presented group $G$ where each edge group is finitely generated and each vertex group is finitely presented with semistable fundamental group at $\infty$. Then $G$ has semistable fundamental group at $\infty$. 
\end{theorem}

Combining Theorems \ref{MainHM} and \ref{split} with the splitting result of Bowditch shows many relatively hyperbolic groups (with boundary cut points) have semistable fundamental group at $\infty$, 
but a broad collection of examples are described near the end of Section 2 of \cite{MS18} that are covered by the Main Theorem of this paper and not by previous results. In particular, for any finitely generated (but not finitely presented) recursively presented group $Q$ and finitely presented group $P$ containing a subgroup isomorphic to $Q$, a finitely presented group $G=A\ast_QP$ is described that is hyperbolic relative to $P$. Here $A$ is finitely generated but not finitely presented. If $P$ has semistable fundamental group at $\infty$ then our Main Theorem \ref{Main} implies $G$ has semistable fundamental group at $\infty$. The techniques of \cite{MS18} break down for such groups.

\section{Semistability Background} \label{SB}

The best reference for the notion of semistable fundamental group at $\infty$ is \cite{G} and we use this book as a general reference throughout this section. While semistability makes sense for multiple ended spaces, we are only interested in 1-ended spaces in this article. Suppose $K$ is a 
locally finite connected CW complex. A {\it ray} in $K$ is a continuous map $r:[0,\infty)\to K$. A continuous map $f:X\to Y$ is {\it proper} if for each compact set $C$ in $Y$, $f^{-1}(C)$ is compact in $X$. Proper rays $r,s:[0,\infty)\to K$ converge to the same end if for any compact set $C$ in $K$, there is an integer $k(C)$ such that $r([k,\infty))$ and $s([k,\infty))$ belong to the same component of $K-C$. 
The space $K$ has {\it semistable fundamental group at $\infty$} if any two proper rays $r,s:[0,\infty)\to K$ that converge to the same end
are properly homotopic  (there is a proper map $H:[0,1]\times [0,\infty)\to X$ such that $H(0,t)=r(t)$ and $H(1,t)=s(t)$). Note that when $K$ is 1-ended, this means that $K$ has semistable fundamental group at $\infty$ if any two proper rays in $K$ are properly homotopic. 
Suppose  $C_0, C_1,\ldots $ is a collection of compact subsets of a locally finite 1-ended complex $K$ such that $C_i$ is a subset of the interior of $C_{i+1}$ and $\cup_{i=0}^\infty C_i=K$, and $r:[0,\infty)\to K$ is proper, then $\pi_1^\infty (K,r)$ is the inverse limit of the inverse system of groups:
$$\pi_1(K-C_0,r)\leftarrow \pi_1(K-C_1,r)\leftarrow \cdots$$
This inverse system is pro-isomorphic to an inverse system of groups with epimorphic bonding maps if and only if $K$ has semistable fundamental group at $\infty$ (see Theorem 2.1 of \cite{M1} or Theorem 16.1.2 of \cite{G}).  It is an elementary exercise to see that semistable fundamental group at $\infty$ is an invariant of proper homotopy type and S. Brick \cite{BR93} proved that semistability is a quasi-isometry invariant. When $K$ is 1-ended with semistable fundamental group at $\infty$, $\pi_1^\infty (K,r)$ is independent of proper base ray $r$ (in direct analogy with the fundamental group of a path connected space being independent of base point).   Theorem 2.1 of \cite{M1} and  Lemma 9 of \cite{M86},  provide several equivalent notions of semistability. Conditions 2 and 3 are the semistability criterion used in the proof of our main theorem. 

\begin{theorem}\label{ssequiv} 
Suppose $K$ is a connected 1-ended  locally finite and simply connected CW-complex. Then the following are equivalent:
\begin{enumerate}
\item Any two proper rays in $K$ are properly homotopic.
\item If $r$ and $s$ are proper rays based at $v$, then $r$ and $s$ are properly homotopic $rel\{v\}$.
\item Given a compact set $C$ in $K$ there is a compact set $D$ in $K$ such that if $r$ and $s$ are proper rays based at $v$ and with image in $K-D$, then $r$ and $s$ are properly homotopic $rel\{v\}$ in $K-C$. 
\end{enumerate}
\end{theorem}

If $G$ is a finitely presented group and $X$ is  a finite connected complex with $\pi_1(X)=G$ then $G$ has {\it semistable fundamental group at} $\infty$ if the universal cover of $X$ has semistable fundamental group at $\infty$. This definition only depends on $G$ (see the proof of Theorem 3 of \cite{LR75} or the opening paragraph of section 16.5 of \cite {G}) and it is unknown if all finitely presented groups have semistable fundamental group at $\infty$.  

\section{Hyperbolicity} \label{Hyp}

There are a number of equivalent forms of hyperbolicity for geodesic metric spaces. In this paper we use the following {\it thin triangles} definition. 

\begin{definition} \label{hyp} 
Suppose $(X,d)$ is a geodesic metric space.  If $\triangle(x,y,z)$ is a geodesic triangle in $X$, let $\triangle '(x',y',z')$ be a Euclidean comparison triangle (i.e. $d'(x',y')=d(x,y)$ etc., where $d'$ is the Euclidean metric.) There is a surjection $f:\triangle '\to \triangle $ which is an isometry on each side of $\triangle'$. The maximum inscribed circle in $\triangle '$ meets the side $[x',y']$ (respectively $[x',z']$, $[y',z']$) in a point $c_z'$ (resp. $c_y'$, $c_x'$) such that

$$d(x',c_z')=d(x',c_y'), \ d(y',c_x')=d(y',c_z'), \ d(z',c_y')=d(z',c_x'). $$

Let $c_x=f(c_x')$, $c_y=f(c_y')$ and $c_z=f(c_z')$. We call the points $c_x,c_y, c_z$ the {\it internal points} of $\triangle $.
There is a unique continuous function $t_{\triangle}:\triangle '\to T_\triangle$ of  $\triangle '$ onto a tripod $T_\triangle$, where $t_\triangle$ is an isometry on the edges of $\triangle '$, and $T_\triangle$ is a tree with one vertex $w$ of degree 3, and vertices $x'', y'', z''$ each of degree one, such that $d(w,z'')=d(z',c_y')=d(z',c_x')$ etc. (See Figure \ref{Fig1}.) 

\begin{figure}
\hspace{-3in}
\vbox to 3in{\vspace {-1in} \hspace {1.7in}
\includegraphics[scale=1]{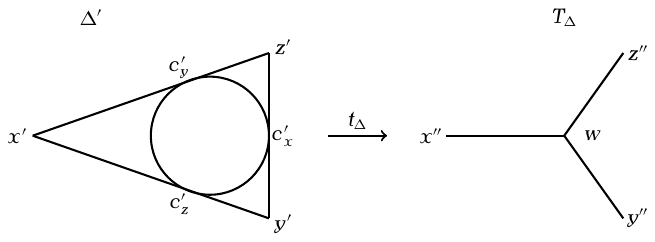}
\vss }
\vspace{-.5in}
\caption{Internal points} 
\label{Fig1}
\end{figure}

\medskip

Let $f_\triangle$ be the composite map $f_\triangle \equiv t_\triangle \circ f^{-1}:\triangle\to T_\triangle$.
We say that $\triangle (x,y,z)$ is $\delta-thin$ if fibers of $f_\triangle$ have diameter at most $\delta$ in $X$. In other words, for all $p,q$ in $\triangle$,
$$f_\triangle (p)=f_\triangle (q) \hbox{ implies } d_X(p,q)\leq \delta .$$
 
The space $X$ is ($\delta$) {\it hyperbolic}  if there is a constant $\delta$ such that all geodesic triangles in $X$ are $\delta$ {\it thin}. 
\end{definition}

In a hyperbolic geodesic metric space $X$ the boundary $\partial X$ can be defined in a number of ways. In Section III.H.3 of \cite{BrHa99} $\partial X$ is defined as the set of equivalence classes $[r]$ of geodesic rays $r$, where $r$ and $s$ are equivalent if there is a number $K\geq 0$ such that $d(r(k),s(k))\leq K$ for all $k\geq 0$. We say $r$ converges to $[r]$. Note that if such a $K$ exists for $r,s$ based at $p$, then our thin triangle condition forces $d(r(k),s(k))\leq \delta$ for all $k\geq 0$. (Simply consider the geodesic triangle formed by $r([0,k+K])$, $s([0,k+K])$ and a geodesic (of length $\leq K$) connecting $r(k+K)$ to $s(k+K)$. The internal points on $r$ and $s$ are beyond $r(k)$ and $s(k)$ respectively.)

If $X$ is a $\delta$ hyperbolic geodesic metric space then there is a metric $d$ on $\partial X$ (induced from an inner product on $X$) such that  $(\partial X,d)$ is compact (see Proposition 3.7 \cite{BrHa99}).  Intuitively, if $r(0)=s(0)$ then $[r]$ is `close' to $[s]$ if $r$ and $s$ fellow travel for a `long' time. 

\section{Cusped Spaces and Relatively Hyperbolic Groups} \label{GM}
Given a finitely generated group $G$ and a collection of finitely generated subgroups ${\bf P}$ of $G$ there are a number of equivalent definitions of what it means for the pair $(G,{\bf P})$ to be relatively hyperbolic or $G$ to be relatively hyperbolic with respect to ${\bf P}$. Theorem \ref{GM3.25} enables us to say the pair $(G,{\bf P})$ is relatively hyperbolic if a certain cusped space is Gromov hyperbolic, so we take this as our definition. The Gromov boundary of this cusped space is the boundary of the pair $(G,{\bf P})$ and is denoted $\partial (G, {\bf P})$. This boundary agrees with the Bowditch boundary of the pair $(G,{\bf P})$.

D. Groves and J. Manning \cite{GMa08} investigate a locally finite space $X$ derived from a finitely generated group $G$ and a collection ${\bf P}$ of finitely generated subgroups. The following definitions are directly from \cite{GMa08}

\begin{definition}  Let $\Gamma$ be any 1-complex. The combinatorial {\it horoball} based on $\Gamma$,
denoted $\mathcal H(\Gamma)$, is the 2-complex formed as follows:

{\bf A)} $\mathcal H^{(0)} =\Gamma (0) \times (\{0\}\cup \mathbb N)$

{\bf B)} $\mathcal H^{(1)}$ contains the following three types of edges. The first two types are
called horizontal, and the last type is called vertical.

(B1) If $e$ is an edge of $\Gamma$ joining $v$ to $w$ then there is a corresponding edge
$\bar e$ connecting $(v, 0)$ to $(w, 0)$.

(B2) If $k > 0$ and $0 < d_{\Gamma}(v,w) \leq 2^k$, then there is a single edge connecting
$(v, k)$ to $(w, k)$.

(B3) If $k\geq 0$ and $v\in \Gamma^{(0)}$, there is an edge joining $(v,k)$ to $(v,k+1)$.

{\bf C)} $\mathcal H^{(2)}$ contains three kinds of 2-cells:

(C1) If $\gamma \subset  \mathcal  H^{(1)}$ is a circuit composed of three horizontal edges, then there
is a 2-cell (a horizontal triangle) attached along $\gamma$.

(C2) If $\gamma \subset \mathcal H^{(1)}$ is a circuit composed of two horizontal edges and two
vertical edges, then there is a 2-cell (a vertical square) attached along $\gamma$. 

(C3) If $\gamma\subset  \mathcal H^{(1)}$ is a circuit composed of three horizontal edges and two
vertical ones, then there is a 2-cell (a vertical pentagon) attached along $\gamma$, unless $\gamma$ is the boundary of the union of a vertical square and a horizontal triangle.
\end{definition}

\begin{definition} Let $\Gamma$ be a graph and $\mathcal H(\Gamma)$ the associated combinatorial horoball. Define a {\it depth function}
$$\mathcal D : \mathcal H(\Gamma) \to  [0, \infty)$$
which satisfies:

(1) $\mathcal D(x)=0$ if $x\in \Gamma$,

(2) $\mathcal D(x)=k$ if $x$ is a vertex $(v,k)$, and

(3) $\mathcal D$ restricts to an affine function on each 1-cell and on each 2-cell.
\end{definition}

\begin{definition} Let $\Gamma$ be a graph and $\mathcal H = \mathcal H(\Gamma)$ the associated combinatorial horoball. For $n \geq 0$, let $\mathcal H_n \subset \mathcal H$ be the full sub-graph with vertex set $\Gamma ^{(0)} \times \{0,\ldots ,N\}$, so that $\mathcal H_n=\mathcal D^{-1}[0,n]$.  Let $\mathcal H^n=\mathcal D^{-1}[n,\infty)$ and $\mathcal H(n)=\mathcal D^{-1}(n)$.   The set $\mathcal H(n)$ is often called a {\it horosphere} or {\it $n^{th}$ level horosphere}. 
\end{definition}

\begin{lemma} \label{GM3.10} 
(Lemma 3.10, \cite{GMa08}) Let $\mathcal H(\Gamma)$ be a combinatorial horoball. Suppose that $x, y \in \mathcal H(\Gamma)$ are distinct vertices. Then there is a geodesic $\gamma(x, y) = \gamma(y, x)$ between $x$ and $y$  which consists of at most two vertical segments and a single horizontal segment of length at most 3.

Moreover, any other geodesic between $x$ and $y$ is Hausdorff distance at most 4 from this geodesic.
\end{lemma}

\begin{definition} Let $G$ be a finitely generated group, let ${\bf P} = \{P_1, \ldots , P_n\}$ be a (finite) family of finitely generated subgroups of $G$, and let $S$ be a generating set for $G$ containing generators for each of the $P_i$.   For each $i\in \{1,\ldots ,n\}$, let $T_i$ be a left transversal for $P_i$ (i.e. a collection of representatives for left cosets of $P_i$ in $G$ which contains exactly one element of each left coset).

For each $i$, and each $t \in T_i$, let $\Gamma_{i,t}$  be the full subgraph of the Cayley graph $\Gamma (G,S)$ which contains $tP_i$. Each $\Gamma_{i,t}$ is isomorphic to the Cayley graph of $P_i$ with respect to the generators $P_i \cap  S$. Then define
$$X(G,{\bf P},S) =\Gamma (G,S)\cup (\cup \{\mathcal H(\Gamma_{i,t})^{(1)} |1\leq i\leq n,t\in T_i\}),$$
where the graphs $\Gamma_{i,t} \subset  \Gamma(G,S)$ and $\Gamma_{i,t} \subset  \mathcal H(\Gamma_{i,t})$ are identified in the obvious way.
\end{definition}
The space $X(G,{\bf P}, S)$ is called the {\it cusped} space for  $G$, ${\bf P}$ and $S$. If $G$ and the $P_i$ have  finite presentations, let  $\mathcal A=\langle S;R\rangle$ be such a presentation that includes sub-presentations of the $P_i$. We add 2-cells to $\Gamma(G,S)$ to form the Cayley 2-complex of this presentation. The resulting expansion of $X(G,{\bf P}, S)$ is  called the cusped space for $G$, ${\bf P}$ and $\mathcal A$ and is denoted $X(G,{\bf P}, \mathcal A)$.  
The next result shows cusped spaces are fundamentally important spaces.  We prove our results in cusped spaces. 

\begin{theorem} \label{GM3.25} 
(Theorem 3.25, \cite{GMa08})
Suppose that $G$ is a finitely generated group and ${\bf P}=\{P_1,\ldots, P_n\}$ is a finite collection of finitely generated subgroups of $G$. Let $S$ be a finite generating set for $G$ containing generating sets for the $P_i$.  A  cusped space $X(G,{\bf P},S)$ is hyperbolic if and only if  $G$ is hyperbolic with respect to ${\bf P}$.
\end{theorem}

Assume $G$ is finitely presented and hyperbolic with respect to the subgroups ${\bf P}=\{P_1,\ldots, P_n\}$ and $S$ is a finite generating set for $G$ containing generating sets for the $P_i$. The $P_i$ and their conjugates are called {\it peripheral} subgroups of $G$. For a finite presentation $\mathcal A$ of $G$ with respect to $S$,  let $Y(\mathcal A)$ be the Cayley 2-complex for $\mathcal A$. So $Y$ is simply connected with 1-skeleton $\Gamma(G,S)$, and the quotient space $G/Y$  has fundamental group $G$. The cusped space $X(G,{\bf P}, S)$ is quasi-isometric to the cusped space $X(G,{\bf P}, \mathcal A)$ and so one is hyperbolic if and only if the other is hyperbolic, and these two spaces have the same boundary. For $g\in G$ and $i\in\{1,\ldots, n\}$ we call $gP_i$ a {\it peripheral coset} in a cusped space. The depth functions on the horoballs over the peripheral cosets extend to $X(G,{\bf P},\mathcal A)$. So that
$$ \mathcal D:X(G,{\bf P},\mathcal A)\to [0,\infty)$$ 
where $\mathcal D^{-1}(0)=Y$ and for each horoball $H$ (over a peripheral coset) we have $H\cap \mathcal D^{-1}(m)=H(m)$, $H\cap \mathcal D^{-1}[0,m]=H_m$ and $H\cap \mathcal D^{-1}[m,\infty)=H^m$. We call each $H^m$ an {\it $m$-horoball}. 

\begin{lemma} \label{geo} (Lemma 3.26, \cite{GMa08}) 
If a cusped space $X$ is $\delta$-hyperbolic, then the $m$-horoballs of $X$ are convex for all $m\geq \delta$. 
\end{lemma} 

Given two points $x$ and $y$ in a horoball $H$, there is a shortest path in $H$ from $x$ to $y$ of the form $(\alpha, \tau,\beta)$ where $\alpha$ and $\beta$ are vertical and $\tau$ is horizontal of length $\leq 3$. Note that if $\alpha$ is non-trivial and ascending and $\beta$ is non-trivial and descending, then $\tau$ has length either 2 or 3. 

If $Y(\mathcal A)$ is the Cayley 2-complex for the finite presentation $\mathcal A$ of the group $G$, then the isometric action of $G$ on $Y$ extends to an isometric action of $G$ on $X(G, {\bf P},\mathcal A)$. This action is height preserving. In the following lemmas, $X=X(G,{\bf P}, \mathcal A)$.

\begin{lemma}\label{tight} (Lemma 5.1, \cite{MS18}) 
Suppose $t_1$ and $t_2$ are vertices of depth $\bar d\geq \delta$ in a horoball $H$ of $X$. Then for each $i\in \{1,2\}$, there is a geodesic $\gamma_i$  from $\ast$ to $t_i$ such that  $\gamma_i$ has the form $(\eta_i, \alpha_i,\tau_i, \beta_i)$, where the end point $x_i$ of $\eta_i$ is the first point of $\gamma_i$ in the horosphere $H(\bar d)$, $\alpha_i$ and $\beta_i$ are vertical and of the same length in $H^{\bar d}$ and $\tau_i$ is horizontal of length $\leq 3$. Furthermore $d(x_1,x_2)\leq 2\delta +1$. 
\end{lemma}

\begin{lemma} \label{close}  (Lemma 2.28, \cite{HM19}) 
Let $P$ be an element of {\bf P}, $g$ be an element of $G$ and $q$ be a closest point of $gP$ to $\ast$ (the identity vertix of $Y$). If $\psi$ is a geodesic from $\ast$ to $gP$ that meets $gP$ only in its terminal point, then the terminal point of $\psi$ is within $6\delta+4$ of $q$. 
\end{lemma}

\begin{lemma}\label{Trans}  (Lemma 4.4, \cite{MS19}) 
Given an integer $K$, there is an integer $A_{\ref{Trans}}(K)$ such that if $\gamma$ is an edge path loop in  $X$ of length $\leq K$, then $\gamma$ is homotopically trivial in $B_{A_{\ref{Trans}}(K)}(v)$ for any vertex $v$ of $\gamma$. 
\end{lemma}

\begin{lemma}\label{track1}  (Lemma 3.3, \cite{HM19}) 
Suppose  $\tilde\lambda=(\lambda, \psi,\bar \lambda)$ is a cusp geodesic  from $x\in qP$ to $y\in qP$ and $d(x,y)\geq 2\delta$. Let $\nu$ be a geodesic in $X$ from $x$ to $y$. Then $|\tilde\lambda |\leq |\nu |+\delta$ and the $i^{th}$ vertex of $\nu$ is within $2\delta$ of the $i^{th}$ vertex of $\tilde \lambda$. If $|\nu|\leq n\leq |\tilde \lambda|$ then the $n^{th}$ vertex of $\tilde \lambda$ is within $2\delta$ of $y$. Finally, the mid point of $\nu$ is an interior point of the geodesic triangle with sides $\nu$, the first half of $\tilde \lambda$ and the second half of $\tilde \lambda$.
\end{lemma}

\section{Filters, Filter Maps and Metrics} \label{FFM}
Our reductions imply the group $G$ is 1-ended as is each $P_i\in {\bf P}$.  
Let $\mathcal A$ be a finite presentation for $G$ that contains finite sub-presentations for the $P_i$. Let $X$ be the cusped space $X(G,{\bf P}, \mathcal A)$.  The compact metric space $\partial (G,{\bf P})$ is path connected and locally path connected (Theorem \ref{lpc}). 
The space $Y\subset X$ is the Cayley 2-complex of $\mathcal A$. For a peripheral coset $gP_i$, let $\Gamma(gP_i)$ be the copy of the Cayley 2-complex of $P_i$ in $Y$ containing $g$.  We use $\hat d$ for our metric on $\partial X=\partial (G,\bf{P})$. 
Any proper ray in $X$ is properly homotopic to a proper edge path ray in the 1-skeleton of $X$. Hence when we show a space has semistable fundamental group at $\infty$ it suffices to show all proper edge path rays are properly homotopic. Since $X$ is quasi-isomorphic to the 1-skeleton of $X$, one is hyperbolic if and only if the other is hyperbolic. Let $d$ be the edge path metric on $X^{(1)}$, the 1-skeleton of $X$. If $A$ is a subcomplex of $X^{(1)}$ let $B_n(A)$ be the neighborhood of radius $n$ about $A$.  For any subcomplex $A$ of $X$, define $St(A)$ to be $A$, union all vertices connected by an edge to a vertex of $A$, union all 2-cells of $X$ all of whose vertices belong to $St(A)$. Define $St_n(A)$ inductively as $St(St_{n-1}(A))$. Note that if $A^{(1)}$ is the 1-skeleton of $A$, then $B_n(A^{(1)})$ is the 1-skeleton of $St_n(A)$. In particular, if for $n\geq 1$, the 1-skeleton of $St_n(v)$ is $B_n(v)$ for all vertices $v$ of $X$. 

\begin{definition}
A {\it filter} $F$ is the realization of a connected graph in $[0,1]\times [0,\infty)$ with the following properties: 

(1) Each vertex is of the form $(t,n)$ for some integer $n\geq 0$ and some $t\in [0,1]$. The points $(0,0)$ and $(0,1)$ are vertices of $F$.

(2) Each edge of $F$ is either vertical or horizontal. A {\it vertical} edge is the convex hull of vertices $(t,n)$ and $(t, n+1)$.
If $(t,n)$ is a vertex of $F$, then $(t,n)$ and $(t,n+1)$ are the vertices of a vertical edge (so every vertex is connected by an edge to exactly one vertex directly above it).

(3) A {\it horizontal} edge is the convex hull of the vertices $(t,n)$ and $(s,n)$ for some integer $n\geq 0$ and numbers $0\leq t<s\leq 1$. 
The horizontal edges at height $n$ form an edge path from $(0,n)$ to $(1,n)$ with consecutive vertices $(0,n), (t_1,n), (t_2,n),\ldots, (1,n)$ where $t_i<t_{i+1}$ for all $i$. (Note that the first coordinates of vertices at height $n$ are a subset of the first coordinates of vertices at height $n+1$.)
\end{definition}

Note that each component of $[0,1]\times [0,\infty)-F$ is a rectangle that is bounded by an edge path loop with exactly two vertical edges, one horizontal edge at height $n$ and all other edges horizontal at height $n+1$. 

The idea is to build filters and proper homotopies that map any vertical edge path in the filter to a geodesic edge path in the 1-skeleton of $X$. Infinitely many of these homotopies will then be combined in a proper way to show that every proper ray in $Y$ is properly homotopic to a certain (nearly geodesic) ray in $Y$ by a proper homotopy in $X_M$ for some fixed integer $M$. 

Let $\delta\geq 1$ be the hyperbolicity constant for $X$.
Given  $\epsilon>0$ there is $N(\epsilon)>0$ such that if $x,y\in \partial X$ and $r_x, r_y$ are geodesic edge path rays at $\ast$ converging to $x$ and $y$ respectively with $d(r_x(N(\epsilon)),r_y(N(\epsilon)))\leq 2\delta+1$ then $\hat d(x,y)\leq \epsilon$. 
Given $N>0$ there is $\epsilon_1(N)$ such that if $x,y\in \partial X$ and $\hat d(x,y)<\epsilon_1 (N)$ then for any geodesics $r_x$ and $r_y$ at $\ast$ converging to $x$ and $y$ respectively, $d(r(N),s(N))\leq 2\delta+1$. 

Since $\partial X$ is compact, connected and locally path connected we have: Given $\epsilon >0$ there is $\rho(\epsilon)>0$ such that if $x,y\in \partial X$ and $\hat d(x,y)\leq \rho(\epsilon)$ then there is a path connecting $x$ and $y$ in $\partial X$ of diameter $\leq \epsilon$. 
Combining these results we have:

\begin{lemma}\label{LC} 
Given an integer $N$ there is an integer $M_{\ref{LC}}(N)>N$ such that if $r$ and $s$ are geodesic edge path rays at $\ast\in X$ (converging to $x, y\in \partial X$ respectively) and $d(r(M_{\ref{LC}}(N)), s(M_{\ref{LC}}(N)))\leq 2\delta+1$, then there is a path $\gamma$ in $\partial X$ from $x$ to $y$ such that for any two points $w_1$ and $w_2$ in the image of $\gamma$ and any geodesic edge paths $q_1$ and $q_2$ at $\ast $ converging to $w_1$ and $w_2$ respectively, $d(q_1(n), q_2(n))\leq \delta$ for all $n\leq N$. 
\end{lemma}

\begin{remark}
The next result provides the primary technical tool to proving our main theorem. It gives an analogue to a geodesic homotopy between two geodesic rays in a CAT(0) space. Suppose $X$ is CAT(0). If $s_0$ and $s_1$ are geodesic rays at $\ast\in X$ and $\gamma$ is a path in $\partial X$ 
from $s_0=\gamma(0)$ to $s_1=\gamma(1)$, then there is a ``geodesic" homotopy $H:[0,1]\times [0,\infty)\to X$ from $s_0$ to $s_1$ defined by $H(a,t)=\gamma(a) (t)$. 
\end{remark}

For technical reasons, we need the following result to apply to edge path rays $s_1'$ and $s_2'$ that are only ``nearly" geodesic. In applications $s_i'$ will be the concatenation of a finite edge path and a geodesic edge path ray. The edge path ray $s_i'$ will synchronously track a geodesic edge path ray. 
\begin{theorem} \label{filter} 
Suppose $K\geq \delta$ is an integer, $s_1$ and $s_2$ are geodesic edge path rays at $\ast$ in $X$ such that $[s_1]\ne [s_2]$, and for $i\in \{1,2\}$, $s_i'$ is an edge path ray such that $d(s_i(t), s_i'(t))\leq K$ for all $t\in [0,\infty)$. Let $\gamma$ be a path in $\partial X$ from $[s_1]=[s_1']$ to $[s_2]=[s_2']$. There is a filter $F(s_1', s_2', \gamma, K)$ for $[0,1]\times [0,\infty)$ and a proper homotopy $f:[0,1]\times [0,\infty)\to X$ (called a filter map for $F$) of  $s_1'$ to $s_2'$ rel $\{\ast\}$, such that:

(1) If $(t,n)$ is a vertex of $F$ with $t\not\in \{0,1\}$ then $f$ restricted to $\{t\}\times [n,\infty)$ is the tail of a geodesic edge path at $\ast\in X$ representing an element of the path $\gamma$ (in $\partial X$).

(2) Each horizontal edge of $F$ is mapped to an edge path of length $\leq K+2\delta$.

(3) If $R$ (an open rectangle) is a component of $[0,1]\times [0,\infty)-F$ and $\alpha$ is the edge path loop bounding the rectangle $R$, then $f(\alpha)$ has image in $B_{2K+\delta+1} (f(v))\subset X$ where $v$ is the upper left 
vertex of $R$. Furthermore $f(R)$ has image in $St_{A_{\ref{Trans}}(2K+\delta+1)}(f(v))$.

(4) If $v=(a,b)$ is a vertex of $F$ and $\tau$ is the vertical segment of $[0,1]\times [0,\infty)$ from $(a,0)$ to $(a,b)$, then $f(\tau)$ and any geodesic from $\ast$ to $f(v)$ will $(2\delta+A_{\ref{Trans}}(2K+\delta+1))$-track one another. 
\end{theorem}

\begin{proof}
By Lemma \ref{Trans} (and adapting to our notation) there is an integer $A(2K+\delta+1)$ such that if $\alpha$ is an edge path loop in $X$ with image in  $B_{ 2K+\delta+1}(v)$ for some vertex $v$ of $\alpha$, then $\alpha$ is homotopically trivial in $St_A(v)$.

We construct the filter $F(s_1', s_2', \gamma,K)$. Choose an integer $N_0\geq 0$ as large as possible such that  for any two points $w_1$ and $w_2$ in the image of $\gamma$ and any geodesic edge paths $q_1$ and $q_2$ at $\ast $ converging to $w_1$ and $w_2$ respectively, $d(q_1(N), q_2(N))\leq \delta$ for all integers $0\leq N\leq N_0$. (Note that $N_0\geq {\delta\over 2}$.) For $j$ an integer between $0$ and $N_0-1$, the only vertices of $[0,1]\times [0,\infty)$ are $(0,j)$ and $(1,j)$. The vertical edges are between $(0,j-1)$ and $(0,j)$, and $(1, j-1)$ and $(1,j)$.  There is a horizontal edge between $(0,j)$ and $(1,j)$.

Next we define $f$ on $[0,1]\times [0,N_0]$ (and on every vertical line above a vertex of $[0,1]\times \{N_0\}$). This process is iterated to define $f$ and the filter $F$. 

(i) $f(t,0)=\ast$ for all $t\in [0,1]$.

(ii) $f(0,t)=s_1'(t)$ and $f(1,t)=s_2'(t)$ for all $t\in[0,\infty)$. 

(iii) For $n$ an integer in $\{1,2,\ldots, N_0-1\}$ let $f$ restricted to the edge $[0,1]\times \{n\}$ be an edge path of length $\leq 2K+\delta$ from $s_1'(n)$ to $s_2'(n)$. (Such a path exists since there is an edge path of length $\leq \delta$ from $s_1(n)$ to $s_2(n)$ and for $i\in \{1,2\}$, edge paths of length $\leq K$ from $s_i'(n)$ to $s_i(n)$). 

(iv) For $n\in \{1,2,\ldots, N_0-1\}$, let $f$ restricted to the rectangle $[0,1]\times [n-1,n]$ be a homotopy in $St_A(s_1'(n))$ (given by Lemma \ref{Trans}) that kills the loop determined by $f$ restricted to the boundary of the rectangle. 

(v) For $k\geq 0$, let $N_k=N_0+k\delta$. Choose points $0=t^0_0<t^0_1<\cdots<t^0_{k(0)}=1$ such that  for any $i$ and two points $u_1$ and $u_2$ in $[t^0_i,t^0_{i+1}]$ and any $\ast$ based geodesic edge paths $q_1\in \gamma(u_1)$ and $q_2\in \gamma(u_2)$, we have  $d(q_1(n), q_2(n))\leq \delta$ for all $n\in [0,N_1]$.  There are $k(0)+1$ vertices $(t^0_0,N_0),(t^0_1,N_0), \ldots (t^0_{k(0)},N_0)$ at level $N_0$ in $F$ and a horizontal edge between $(t^0_j,N_0) $ and $(t^0_{j+1}, N_0)$ for each $j$. For each $n\in \{1,\ldots, k_0-1\}$ add a vertical edge path ray $\{t_n^0\}\times [n_0,\infty)$ to $F$ (with vertices $(t_n^0,n)$ for each integer $n\geq n_0$.) Let $r_n^0$ be a geodesic edge path at $\ast$ converging to $\gamma(t^0_n)$. Let $r_0^0=s_1'$ and $r_{k(0)}^0=s_2'$. For $n\in \{0,\dots, k(0)\}$ and $a\in [N_0,\infty)$ define $f(t^0_n,a)=r_n^0(a)$. (This agrees with our earlier definition of $f$ on $\{0,1\}\times [N_0,\infty)$.

Note that for $n\in \{1,\ldots, k(0)-2\}$, $d(f(t_n^0,N_0), f(t_{n+1}^0,N_0))\leq \delta$ and for $n\in \{0,k(0)-1\}$,  $d(f(t_n^0,N_0), f(t_{n+1}^0,N_0))\leq K+\delta$.
For $n\in \{1,\ldots, k(0)-2\}$ define $f$ restricted to the edge between $(t_n^0, N_0)$ and $(t_{n+1}^0, N_0)$ to be an edge path of length $\leq \delta$. For $n\in \{0, k(0)-1\}$ define $f$ restricted to the edge between $(t_n^0, N_0)$ and $(t_{n+1}^0, N_0)$ to be an edge path of length $\leq K+\delta$. (see Figure \ref{F2}).

\begin{figure}
\hspace{-3in}
\vbox to 3in{\vspace {-1in} \hspace {.8in}
\includegraphics[scale=1]{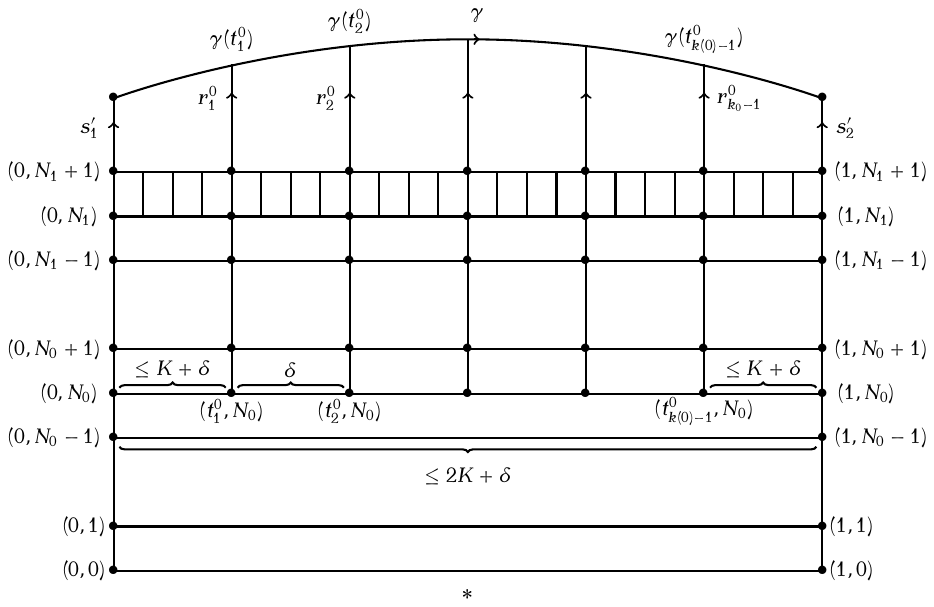}
\vss }
\vspace{1.6in}
\caption{A Filter Homotopy} 
\label{F2}
\end{figure}

\medskip

Recall  that  $d(s_1(N_0), s_1'(N_0))\leq K$, $f(0,N_0)=s_1'(N_0)$, and for all $n\in \{0,1,\ldots k(0)-1\}$,  $d(s_1(N_0), f(t_n^0, N_0))\leq \delta$. 
Hence $d(f(0,N_0), f(t_n^0,N_0))\leq K+\delta$ for all $n$. The edge path loop bounding the rectangle $[0,1]\times [N_0-1, N_0]$ is mapped by $f$ to an edge path loop in $St_{2K+\delta+1}(f(0,N_0))$ (recall $K\geq \delta$).  This loop is homotopically trivial in $St_A(f(0,N_0))$. Extend $f$ to the rectangle by this homotopy.

Iterate this process on each of the regions $[t^0_n, t^0_{n+1}]\times [N_0,N_1]$ for $n\in \{0,1,\ldots, k_0-1\}$. 
This extends $f$ to $[0,1]\times [0,N_1]$ and each vertical ray above a vertex of $[0,1]\times \{N_1\}$. Repeated iterations defines a filter $F$ and a proper homotopy/filter map on $[0,1]\times [0,\infty)$.

(proof of part (4)): Again, let $A=A_{\ref{Trans}}(2K+\delta+1)$. Say $b=N_i-j$ where $i\geq 1$ and $1\leq j\leq\delta$ (if $j=\delta$ then $b=N_{i-1}$). A terminal  segment of $\tau$ (see (4)) is the vertical segment of $F$ from $(a, N_{i-1})$  to $(a,b)$. There are integers $i$ and $k$ such that $a=t^{i-1}_k$ (and $(t^{i-1}_k, N_{i-1})$ is a vertex of the subdivision of the horizontal segment $[0,1]\times \{N_{i-1}\}$).  The geodesic edge path ray $r_k^{i-1}$ at $\ast$ in $X$ is such that $r_k^{i-1} (t)=f((t_k^{i-1} ,t))$ (where again $a=t_k^{i-1}$) for $t\geq N_{i-1}$. (See Figure \ref{F3})

\begin{figure}
\hspace{-3in}
\vbox to 3in{\vspace {-1in} \hspace {1.5in}
\includegraphics[scale=1]{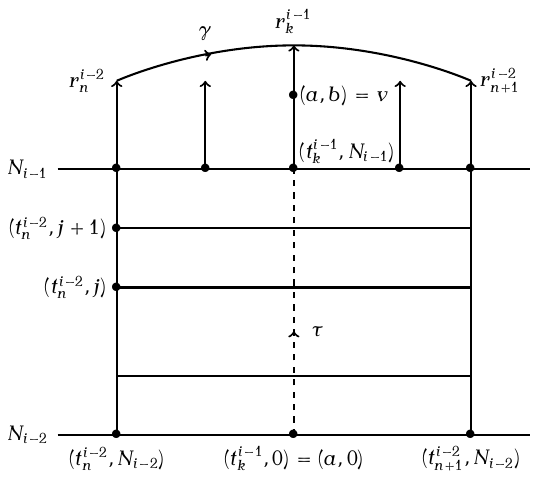}
\vss }
\vspace{.8in}
\caption{Tracking Geodesics} 
\label{F3}
\end{figure}

\medskip

We will show that $f(\tau)$ and $r_k^{i-1}|_{[0,b]}$ will $(2\delta +A)$-track one another. We already have that $f(\tau(t))=f(t_k^{i-1}, t)=r_k^{i-1}(t)$ for  $t\in [N_{i-1}, b]$ (a terminal segment of $f(\tau)$). Choose $n$ such that $t_k^{i-1}\in [t^{i-2}_n,t^{i-2}_{n+1}]$. By construction, the rays $r_n^{i-2}$ and $r_{k}^{i-1}$ will $\delta$ fellow travel on $[0,N_{i-1}]$. By (3), for any rectangle $R=[t^{i-2}_n,t^{i-2}_{n+1}]\times [j,j+1]$ (for $j$ an integer in $[N_{i-2}, N_{i-1}-1]$), we have $f(R)\subset St_A(r_n^{i-2}(j+1))$. In particular, 
$f(\tau)$ and $r_n^{i-2}$ will $A$-track one another on $[N_{i-2},N_{i-1}]$.
Since $r_n^{i-2}$ and $r_k^{i-1}$ will $\delta$-track one another on $[0,N_{i-1}]$,
$f(\tau)$ and $r_k^{i-1}$ will $(\delta +A)$-track one another on $[N_{i-2},N_{i-1}]$.
Next find $p$ such that $t_k^{i-1}$ is between $t^{i-3}_p$ and $t^{i-3}_{p+1}$ and repeat the argument on $[N_{i-3},N_{i-2}]$ and subsequent intervals to obtain $f(\tau)$ and $r_k^{i-1}$ will $(\delta +A)$-track one another on $[0,b]$.
Now $r_k^{i-1}|_{[0,b]}$ and any other geodesic from $\ast$ to $f(v)$ will $\delta$-track one another, completing the proof of (4). 
\end{proof} 

\section{Triangulations and Simplicial Homotopies}\label{TSH}

In this section we define a triangulation of $X$ that respects the action of $G$. Given a filter $F$ and filter map $f_1:[0,1]\times [0,\infty)\to X$, we produce a triangulation for $[0,1]\times [0,\infty)$ and a proper simplicial map $f:[0,1]\times [0,\infty)\to X$ that agrees with $f_1$ on $F$.

 Our primary tool is E. C. Zeeman's relative simplicial approximation theorem. We follow Zeeman's notation. 
 
If $K$ is a simplicial complex, let $|K|$ denote the polyhedron underlying $K$ (also called the {\it realization} of $K$). If $L$ is a subcomplex of $K$, let $(K\  mod\  L)'$ denote the {\it barycentric derived complex of $K$ modulo $L$} which is obtained from $K$ by subdividing barycentrically all simplexes of $K-L$ in some order of decreasing dimension. Note that  $L$ is a subcomplex of $(K\  mod\  L)'$. Inductively define 
 $$K_0=K,$$
 $$K_r=(K_{r-1}\  mod\  L)'.$$
In 1964, E. C. Zeeman proved The Relative Simplicial Approximation Theorem. 

\begin{theorem} \label{Z} (Main Theorem, \cite{Z64}) 
Let $K$, $M$ be finite simplicial complexes and $L$ a subcomplex of $K$. Let $f:|K|\to |M|$ be a continuous map such that the restriction $f|_L$ is a simplicial map from $L$ to $M$. Then there exists an integer $r$, and a simplicial map $g:K_r\to M$ such that $g|_L=f|_L$ and $g$ is homotopic to $f$ keeping $L$ fixed. 
\end{theorem}

First a we construct a triangulation of $X$. Recall that $\mathcal A$ is a finite presentation for $G$ and $\mathcal A$ contains a finite presentation for each $P\in {\bf P}$, as a subpresentation. Each 2-cell of $Y$ is bounded by an edge path (corresponding to a relation of our presentation $\mathcal A$ of $G$). In each 2-cell $E$ add a vertex $v$ ($G$-equivariantly) and  an edge from $v$ to each vertex of the boundary of $e$. Triangles are formed (in the usual way) from the two vertices of an edge in the boundary of $E$ and  $v$. This triangulates $E$ unless its boundary has length 2 (there may be a generator of order 2). In this case, add a vertex to each edge of $E$, a vertex $v$ to $E$ and add an edge from $v$ to each vertex in the boundary of $E$. This is done respecting the action of $G$ on $Y$ and 
gives a triangulation of $Y$. If $E$ is a 2-cell of a horoball {\bf H} and $E$ has three horizontal edges in its boundary, then $E$ is a triangle of our triangulation. If $E$ has two vertical edges and two horizontal edges, then add a single diagonal edge to $E$. For each translate $gE$ add a diagonal edge that respects the action of $G$. If $E$ has two vertical edges and 3 horizontal edges, let $v$ be the common vertex of the two lower horizontal edges. Add edges from $v$ to the two vertices of $E$ that are one level above $v$. 
In this way no additional vertices are added to any horoball of $X$ and we have a triangulation of $X$ that respects the action of $G$. 

Next suppose $F(s_1',s_2', \gamma, K)$ is a filter and $f_1:[0,1]\times [0,\infty)\to X$ is a filter map for $F$. The vertices of $F$ are called {\it filter vertices}. If $e$ is a horizontal edge of $F$ and $f_1(e)$ is an edge path (of length $\leq 2K+\delta$), then add vertices to $e$ (and replace $e$ by the corresponding edges) so that $f_1$ is simplicial  on $e$. These new vertices are not called filter vertices. At this point, $f_1$ is simplicial on our triangulation of $F$, but we have not dealt with 2-cells yet. Suppose $R$ is a rectangle of $[0,1]\times [0,\infty)-F$. Add a vertex $w$ to $R$ and an edge from $w$ to each vertex of the boundary of $R$ in order to triangulate $\bar R$ (the closure of $R$). Recall that $f_1$ restricted to $\bar R$ is a homotopy that kills the boundary loop of $R$ in $St_{A_{\ref{Trans}} (2K+\delta+1)}(f_1(v))$ where $v$ is the upper left (filter) vertex of $R$. Let $f|_{\bar R}$ be a simplicial approximation to $f_1$ with image in (our triangulated) $St_{A_{\ref{Trans}} (2K+\delta+1)}(f_1(v))$ such that $f$ agrees with $f_1$ on the boundary of $R$. We have shown:

\begin{lemma}\label{Est} 
Suppose $F(s_1',s_2',, \gamma, K)$ is a filter and $f_1:[0,1]\times [0,\infty)\to X$ is a filter map for $F$. There are triangulations of $X$ and $[0,1]\times [0,\infty)$ and a simplicial map $f:[0,1]\times [0,\infty)\to X$ that agrees with $f_1$ on $F$. Furthermore, for any rectangle $R$ of $[0,1]\times [0,\infty)-F$, $f(\bar R)\subset St_{A_{\ref{Trans}} (2K+\delta+1)}(f(v))$ where $v$ is the upper left (filter) vertex of $\bar R$. In particular, if $w$ is a vertex of $\bar R$, (in our triangulation of $[0,1]\times[0,\infty)$  and $\phi$ is an edge path in $\bar R$ from $w$ to the upper left (filter) vertex $v$ of $\bar R$, then 
$$d(f(v), f(w))\leq A_{\ref{Trans}}(2K+\delta+1)$$
and $f(\phi)$ is an edge path from $f(w)$ to $f(v)$ such that 
$$im(f(\phi))\subset B_{A_{\ref{Trans}} (2K+\delta+1)}(f(v))\subset B_{2A_{\ref{Trans}} (2K+\delta+1)}(f(w)).$$ 
Hence if $E\geq 0$ and $\mathcal D(f(w))> 2A_{\ref{Trans}} (2K+\delta+1)+E$ then the image of $\phi$ is in the horoball containing $f(w)$ and $\mathcal D(f(v))> A_{\ref{Trans}} (2K+\delta+1)+E$.
\end{lemma}

While more general projections are considered in \cite{MS18} we are only interested in projecting proper edge path rays of $X_K$ into $Y$. In fact, we need only consider special projections obtained by projecting the individual horizontal edges of a ray into $Y$. 

Suppose $K\geq 0$ and $e=(v,w)$ is an edge in $X(K)$. Say $\tau$ is the vertical edge path from $Y$ to $v$ and $\bar \tau$ is the vertical edge path from $Y$ to $w$. Then $\gamma$ is a {\it projection} of $e$ (or $(\tau, e,\bar \tau^{-1})$) to $Y$  if $\gamma$ is a shortest edge path in $Y$ from the initial point $\tau$ to the initial point of $\bar \tau$. If $r$ is an edge path in $X_K$ with initial and end point in $Y$ or an edge path ray in $X_K$ with initial point in $Y$, then $\hat r$ is a  projection of $r$ to $Y$ if $\hat r$ is obtained from $r$ by replacing each horizontal edge $e$ of $r$ by a projection of $e$ to $Y$.
Suppose $K>0$ is an integer and $r$ is a proper edge path ray in $X_K$ with initial point in $Y$. We construct a proper simplicial homotopy $H$ from $r$ to a projection of $r$ into $Y$ such that the image of $H$ is in $St_{K +1}(im(r))$.
 The following is a special case of Lemma 5.6 of \cite{MS18}
 
\begin{lemma}\label{Proj} 
Suppose $e$ is an edge of $\bar H(K)$ for some integer $K>0$. If $\gamma$ is a projection of $e$ into $Y$ then each vertical line at a vertex of $\gamma$ passes within $1$ horizontal unit of a vertex of  $e$.
\end{lemma}

\begin{lemma}\label{ProjH} 
Suppose that $r$ is a proper edge path ray at $v\in Y$. Also assume that  $r$ has image in $X_{K}$ for some integer $K\geq 0$.  Then a projection of $r$ to $Y$ is properly homotopic rel$\{v\}$ to $r$ by a proper simplicial homotopy with image in $St_{K+1}(im(r))$.
\end{lemma}

\begin{proof}
If $e$ is a horizontal edge $e$ of $r$ then consider $(\tau, e, \bar\tau^{-1})$ where $\tau$ (respectively $\bar \tau$) is vertical from $Y$ to the initial (respectively terminal) point of $e$. It suffices to show that $(\tau, e, \bar \tau^{-1})$ is homotopic to a projection $Qe$ of $e$ by a simplicial homotopy in $St_{K+1}(e)$.  Suppose $e=(a,b)$ where $a$ and $ b$ are vertices of height $\leq K$.  Let $\gamma_0$ be a shortest path in $Y$ from the initial  point of $\tau$ to the initial point of $\bar \tau$. If $\gamma_0$ is the edge path $(e_1,e_2,\ldots, e_n)$ then there is a vertical pentagon with base $(e_1,e_2)$, two vertical sides and a horizontal edge $e_1^1$ at level 1. Let $H_1^1$ be the obvious simplicial homotopy of $(e_1,e_2)$ to $(b_1, e_1^1, b_3^{-1})$ where the $b_i$ are vertical edges. Construct $H_1^2$ a simplicial homotopy of $(e_3,e_4)$ to $(b_3, e_2^1,b_5^{-1})$. Continuing, the last homotopy may have base $(e_{n-1}, e_n)$ (if $n$ is even) or just $e_n$ (if $n$ is odd). Combining these homotopies gives a simplicial homotopy $H_1$ of $\gamma_0=(e_1,\ldots , e_n)$ to $(b_1,\gamma_1, b_{n+1}^{-1})$ where $\gamma_1=(e_1^1, e_2^1,\ldots)$ is  horizontal of length $\leq {n\over 2}+1$. Similarly define $H_2$ a simplicial homotopy of $\gamma_1$ to a vertical edge followed by the horizontal edge path $\gamma_2$ followed by another vertical edge.  Continuing this process, we find that the last homotopy is one with base of length one or two. Hence the top edge is $e$. Combining these simplicial homotopies gives a simplicial homotopy of $ \gamma_0$ to $(\tau ,e,\bar\tau^{-1})$. Lemma \ref{Proj} implies that the image of this homotopy is in $St_{K+1}(e)$. \end{proof}

\section{Preliminary Results}\label{prelim}
In order to build certain ideal triangles, we need a geodesic line in $X_{19\delta}$.

\begin{theorem} \label{GM3.33} 
There is an infinite order element $g\in G$ so that if $\rho$ is a geodesic in $X$ from $\ast$ to $g\ast$, then the line $l=(\ldots, g^{-1}\rho, \rho, g\rho, g^2\rho,\ldots)$ is a bi-infinite geodesic that has image in $\mathcal D^{-1}([0,19\delta])$. 
\end{theorem}

\begin{proof}
By Theorem 3.33 \cite{GMa08} there is a geodesic line $\ell$ in $\mathcal D^{-1}([0,19\delta])$ and an infinite order element $g_1\in G$ such that $g_1\ell=\ell$. Certainly the image of $\ell$ is not a subset of a horoball and so $\ell$ must contain a vertex $v=h\ast$ (for $h\in G$) of $Y$.  The element $g=h^{-1}g_1h$ stabilizes the geodesic line $h^{-1}\ell$ (containing $\ast$). If $\rho$ is the subgeodesic of $h^{-1}\ell$ from $\ast$ to $g\ast$, then $h^{-1} \ell=(\ldots, g^{-1}\rho, \rho, g\rho, g^2\rho,\ldots)$. Since $G$ is height preserving, $h^{-1} \ell$ has image in $\mathcal D^{-1}([0,19\delta])$. 
\end{proof}

Let $\ell^+$ be the geodesic ray $(\rho, g\rho, g^2\rho,\ldots)$ at $\ast$ and let $\ell^-$ be the geodesic ray $(g^{-1}\rho^{-1}, g^{-2}\rho^{-1}, g^{-3}\rho^{-1},\ldots)$ at $\ast$ (so that $\ell^+$ and $\ell^{-}$ determine the two ends of $\ell$). 
Let $v$ be a vertex of $Y$ (so that $v\in G$) and consider an ideal geodesic triangle determined by the geodesic line $v\ell$ and two geodesic rays $s_v^+$ and $s_v^-$ at $\ast$, where $s_v^+$ (respectively $s_v^-$) converges to the same point of $\partial (G, {\bf P})$ as does $v\ell^+$ (respectively $v\ell^-$). This implies that $v$ is within $\delta$ of either $s_v^+$ or $s_v^-$. In the former case let $r_v$ be $v\ell^+$, otherwise let $r_v$ be $v\ell^-$. We have:

\begin{lemma}\label{track} 
The geodesic $r_v$ at $v$ is either $v\ell^+$ or $v\ell^-$. If $\alpha_v$ is a geodesic from $\ast$ to $v$ and $s_v$ is the geodesic ray at $\ast$ such that $[r_v]=[s_v]\in \partial X$, then for each integer $n\geq 0$ the vertex $s_v(n)$ is within $\delta$ of the $n^{th}$ vertex of $(\alpha_v, r_v)$.
\end{lemma}

The following definition and theorem were critical components in the proof of the homology version of our main theorem. They play an important role in this paper.

\begin{definition}
We call the pair $(E,\alpha)$ a {\it disk pair} in the simplicial complex $[0,\infty)\times [0,1] $ if $E$ is an open subset of $[0,\infty) \times [0,1]$ homeomorphic to $\mathbb R^2$, $E$ is a union of (open) cells, $\alpha$ is an embedded edge path bounding $E$ and $E$ union $\alpha$ is a closed subspace of $[0,\infty)\times [0,1]$ homeomorphic to a closed ball or a closed half space in $[0,\infty)\times [0,1]$. When $\alpha$ is finite, we say the disk pair is finite, otherwise we say it is unbounded. 
\end{definition} 

We will apply the next result with $X$ equal to the cusped space for $(G, \mathcal P)$, $Y$ equal to the Cayley 2-complex  of $(G,\mathcal A)$ in $X$ and the $Z_i$ being the $G$-translates of the $\Gamma(P_i)$ in $X$. The set $X-Y$ will be the union of the open horoballs above the $Z_i$. This result will allow us to start with a proper simplicial homotopy $M:([0,1]\times [0,\infty),[0,1]\times \{0\})\to  (X,\ast)$ of proper edge path rays  $r$ and $s$ (with images in $Y$) and ``excise" certain  parts of $[0,1]\times [0,\infty)$ not mapped into $Y$. When the homotopy is built primarily from a filter, we will be able to replace it by a proper homotopy between $r$ and $s$ with image completely in $X_N$ for some integer $N$.
 
\begin{theorem} \label{excise} (Theorem 6.1, \cite{MS19}) 
Suppose 
$$M:([0,1]\times [0,\infty),[0,1]\times \{0\})\to  (X,\ast)$$ 
is a proper simplicial homotopy $rel\{\ast\}$ of proper edge path rays  $r$ and $s$ into a connected 
locally finite simplicial 2-complex $X$, where $r$ and $s$ have image in a subcomplex $Y$ of $X$. Say $\mathcal Z=\{Z_i\}_{i=1}^\infty$ is a collection of connected subcomplexes of $Y$ such that only finitely many $Z_i$ intersect any compact subset of $X$. Assume that each vertex of $X-Y$ is separated from $Y$ by exactly one $Z_i$. 

Then there is an index set $J$ such that for each $j\in J$, there is a disk pair $(E_j,\alpha_j)$ in  $[0,1]\times [0,\infty)$ where the $E_j$ are disjoint, $M$ maps $\alpha_j$  to $Z_{i(j)}$ (for some $i(j)\in \{1,2,\ldots\}$) and $M([0,1]\times [0,\infty)-\cup_{j\in J} E_j)\subset Y$. 
\end{theorem} 

\section{The Proof of the Main Theorem} \label{MThm} 
In this section we prove there is an integer $M_0$ such that $X_{M_0}$ has semistable fundamental group at $\infty$. Since $X_{M_0}$ is simply connected and $G/X_{M_0}$ is a finite complex. By definition (see \S \ref{SB}) $G$ has semistable fundamental group at $\infty$ if and only if $X_{M_0}$ has semistable fundamental group at $\infty$.

Recall that for each vertex $v$ in $Y$ we have defined the geodesic edge path ray $r_v$ at $v$. If $v\in gP_i$ let $Qr_v$ be some projection of $r_v$ into $Y$. The next lemma is the key technical fact of the paper. All homotopies that appear following this lemma are derived from homotopies guaranteed by this lemma. 

\begin{lemma} \label{periEdge} 
Let $M_0=2A_{\ref{Trans}}(7\delta+1)+2\delta +1$. If $e=(v,w)$ is an edge of a peripheral coset $gP_i$ and $d=(w,q)$ is an edge of $Y$ then:

\noindent (1) The edge path ray $Qr_v$ is properly homotopic rel$\{v\}$ to both the edge path ray $(e,Qr_w)$  and to the edge path ray $(e,d,Qr_q)$, by homotopies in $X_{M_0}$.

\medskip

\noindent (2) For $N>0$ there is $M_{\ref{periEdge}}(N)>N$ such that if 
$\{v,w\}\cap B_{M_{\ref{periEdge}}(N)}(\ast )=\emptyset$ 
then there is an edge path $\psi$  in $gP_i$ from $v$ to $w$ such that  $Qr_v$ is properly homotopic rel$\{v\}$ to $(\psi, Qr_w)$ in $X_{M_0}-B_N(\ast)$. 

If $q\not\in  B_{M_{\ref{periEdge}}(N)}(\ast )$ then there is an edge path
$\psi'$  in $gP_i$ from $v$ to $w$ such that  $Qr_v$ is properly homotopic rel$\{v\}$ to $(\psi', d,Qr_q)$  in $X_{M_0}-B_N(\ast)$.

Furthermore, if $gP_i\cap B_N(\ast)=\emptyset$ then we may take $\psi=\psi'=e$. 
\end{lemma}

\begin{proof} 
We prove $Qr_v$ is properly homotopic to $(e,Qr_w)$ and $(\psi, Qr_w)$ in parts ( 1) and (2) of the Lemma. The proof that $Qr_v$ is properly homotopic to $(e,d,Qr_q)$ and $(\psi', d, Qr_q)$ is completely analogous to that argument, with $Qr_w$ simply replaced by $(d,Qr_q)$. 
 
We begin by proving part (2) of the lemma. Part (1) has an analogous, but more elementary proof that we include at the end.
Let $\mathcal A=\{A_1,A_2, \ldots, A_m\}$ be the set of peripheral cosets that intersect $B_N(\ast)$.
Choose $K_1>N+19\delta_1+1$ such that for $j\in\{1,\ldots, m\}$ and $v_j$ a closest vertex of $A_j$ to $\ast$, we have $B_{6\delta+4}(v_j)\subset B_{K_1}(\ast)$. 
For $j\in\{1,\ldots, m\}$ let $H^j$ be the horoball over $A_j$. 
Let $c_j$ be a closest point of $H^j(\delta)$ to $\ast$. Let $L=A_{\ref{Trans}}(7\delta+1)$. Assume that for $j\in \{1,\ldots, m\}$
 $$B_{2L+5\delta +3}(c_j)\subset B_{K_1}(\ast).$$

We fix the following constants:
$$L=A_{\ref{Trans}}(7\delta+1);\ \ \ K_2=3K_1+2L+16\delta+3;\ \ \ M=M_{\ref{LC}}(K_2).$$ 

Note that $M$ depends only on $N$. There are two Cases. We will show that if $gP_i=A_j$ for some $j\in \{1,\ldots, m\}$, then $M=M_{\ref{periEdge}}(N)$ satisfies the second conclusion of our lemma. If $gP_i\ne A_j$ for any $j\in \{1,\ldots, m\}$, then  a different value for $M_{\ref{periEdge}}(N)$ satisfies the conclusion of the lemma. We finish our proof by choosing $M_{\ref{periEdge}}(N)$ to be the large of the two.

Recall that by Lemma \ref{Trans}, if $\alpha$ is an edge path loop in $X$ with image in  $B_{ 7\delta+1}(v)$ for some vertex $v$ of $\alpha$, then $\alpha$ is homotopically trivial in $St_{A_{\ref{Trans}}(7\delta+1)}(v)=St_L(v)$.

\medskip

\noindent {\bf Case 1.} Assume that $gP_i=A_1$ and $(v,w)$ is in $X-B_M(\ast)$.
 
\medskip
 
\noindent Say our edge path ray $r_v$ (at $v$) converges to $x\in \partial X$ and $r_w$ converges to $y\in \partial X$. 
By Lemma \ref{track}, the vertex of $s_v$ that is $d(v,\ast)$ from $\ast$ is within $\delta$ of $v$. Similarly for $s_w$. Since these points of $s_v$ and $s_w$ are within $2\delta+1$ of one another, we have $d(s_v(M),s_w(M))\leq 2\delta+1$.  
By Lemma \ref{LC}, there is a path $\gamma$ in $\partial X$ from $x$ to $y$ such that for any two points $w_1$ and $w_2$ in the image of $\gamma$ and any geodesic edge paths $q_1$ and $q_2$ at $\ast $ converging to $w_1$ and $w_2$ respectively, we have $d(q_1(k), q_2(k))\leq \delta$ for all $k\leq K_2=3K_1+2L +16\delta+3$. (See Figure \ref{F4}.)

\begin{figure}
\hspace{-3in}
\vbox to 3in{\vspace {-1in} \hspace {1.5in}
\includegraphics[scale=1]{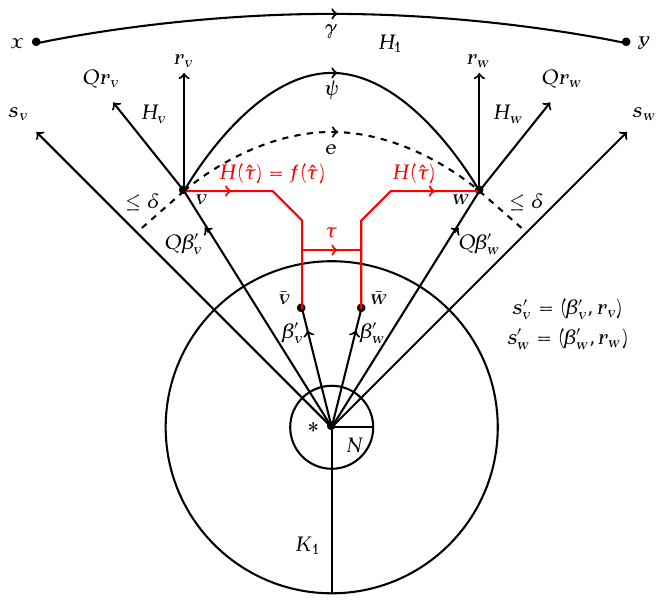}
\vss }
\vspace{1.5in}
\caption{Building Proper Homotopies} 
\label{F4}
\end{figure}

\medskip

Suppose $\beta_v$ is a geodesic from $\ast$ to $v$ and $\bar v$ is the first point of $\beta_v$ in $gP_i$.  Since $gP_i=A_1$ and $v_1$ is a closest point of $A_1$ to $\ast$, Lemma \ref{close}  implies $d(\bar v,v_1)\leq 6\delta+4$ and so
$$\bar v\in B_{K_1}(\ast).$$ 
Let $\beta_v'$ be the edge path from $\ast$ to $v$ obtained from $\beta_v$ by replacing the segment of $\beta_v$ from $\bar v$ to $v$ by a cusp geodesic and note that this cusp geodesic has length $\geq 2K_1+2L +16\delta+3$. Let $s_v'$ be $\beta_v'$ followed by $r_v$. Similarly define $s_w'$. Lemma \ref{track} implies the $n^{th}$ vertex of $s_v$ is within $\delta$ of the $n^{th}$ vertex of $(\beta_v, r_v)$ for all $n$.
Lemma \ref{track1} implies that for all $n\geq 0$
$$d(s_v(n),s_v'(n))\leq 3\delta .$$ 
Similarly for $s_w$ and $s_w'$. Let $F=F(s_v',s_w',\gamma, 3\delta)$ and $f:[0,1]\times[0,\infty) \to X$ be the filter and filter map of Theorem \ref{filter} (so that the constant $K$ of Theorem \ref{filter} is $3\delta$). 

At this point the argument becomes technical.  We give a brief outline of the Case 1 argument and refer the reader to Figure \ref{F4}.  We construct a proper homotopy $H_v$ between $r_v$ and the projection $Qr_v$. Similarly with $H_w$. Then take simplicial approximations of $H_v$, $H_w$ and $f$. Then we combine these three proper simplicial homotopies with simplicial homotopies of $\beta_v'$ to $Q\beta_v'$ and $\beta_w'$ to $Q\beta_2'$. This gives a proper simplicial homotopy $H$ between $Qs_v'$ and $Qs_w'$. 
Apply Theorem \ref{excise} to $H$. We will show there is a disk pair $(D,\alpha)$ such that both $v$ and $w$ are vertices of $H(\alpha)$. In order to do this, we show there is a path $\hat \tau$ in the domain of $H$ that is mapped by $H$  to a path connecting $v$ and $w$ into the horoball above $A_1$ (so $\hat \tau$ belongs to the disk $D$ of a disk pair $(D,\alpha)$).  
The path $H(\hat \tau$) is represented in Figure \ref{F4} by a red path 
from $v$ to $w$ with $\tau$ as a subpath.
If $\hat \psi$ is the part of $\alpha$ above $\hat \tau$ then $f(\hat\psi)$ is the path $\psi$ of our lemma and we will only use the part of $H$ that lies above $\hat\psi$ to obtain our final homotopy. Other disks of disk pairs of Theorem \ref{excise} are also removed, but we will show only finitely many can have boundary path in a given peripheral coset and none of these peripheral cosets will in $\mathcal A$. 
If $\alpha'$ is such a boundary path and $\alpha'$ is finite then we extend our homotopy to the disk it bounds by an arbitrary homotopy that kills $\alpha'$ in the corresponding Cayley 2-complex of its peripheral coset. If $\alpha'$ is unbounded then we extend our homotopy to the disk (halfspace) it bounds by a proper homotopy of two opposite rays forming $\alpha'$ in the corresponding Cayley 2-complex. This gives the proper homotopy described by part (2) of the lemma (and completes the outline). 

Recall that in the proof of Lemma \ref{filter}, the number $N_0\geq 0$ was chosen large as possible such that  for any two points $w_1$ and $w_2$ in the image of $\gamma$ and any geodesic edge paths $q_1$ and $q_2$ at $\ast $ converging to $w_1$ and $w_2$ respectively, $d(q_1(n), q_2(n))\leq \delta$ for all integers $n\in [0,N_0]$.
Since $M=M_{\ref{LC}}(K_2)$
$$N_0\geq K_2=3K_1+2L +16\delta+3.$$

As noted earlier, the cusp geodesic from $\bar v$ to $v$ has length at least $2K_1+2L +16\delta+3$ and so its  initial vertical segment has length at least $K_1+L+8\delta$. 

Let $D_1=d(\bar v,\ast)$ so that $0\leq D_1\leq K_1$. Then $s_v'(D_1+L+K_1+8\delta)$ has depth $L+K_1+8\delta$ in the horoball over $gP_i$.  In the construction of the filter in Lemma \ref{filter}, for each integer $n<N_0$ there was an edge from $(0,n)$ to $(1,n)$.  Since $D_1\leq K_1$, we have $D_1+L+K_1+8\delta<N_0$. This implies there is an edge in our filter from $(0, D_1+L+K_1+8\delta)$ to $(1,D_1+L+K_1+8\delta)$. The image under $f$ of this edge is an edge path $\tau$ of length $\leq 2(3\delta)+\delta =7\delta$ from $s_v' (D_1+L+K_1+8\delta)$ to $s_w'(D_1+L+K_1+8\delta)$. Since $s_v' (D_1+L+K_1+8\delta)$ has depth $L+K_1+8\delta$, the path $\tau$ (of length $\leq 7\delta$) has image in the horoball over $gP_i$. This implies $s_w'(D_1+L+K_1+8\delta)$ is a point of the cusp geodesic from $\bar w$ to $w$ (since $\bar w$ is the first point of $s_w'$ in that horoball). 

\medskip

 \noindent {\bf (A)} 
Let $D_v$  be the length of the subpath of $s_v'$  from $\ast$ to $v$. Similarly define $D_w$. Let $\hat \tau$ be the edge path in the filter from $(0, D_v)$ to $(0, D_1+L+K_1+8\delta)$ followed by the edge from $(0, D_1+L+K_1+8\delta)$ to $(1, D_1+L+K_1+8\delta)$ followed by the edge path from $(1, D_1+L+K_1+8\delta)$ to $(1,D_w)$. The path $f(\hat \tau)$  follows our cusp geodesic from $v$ to the initial point of $\tau$, then follows $\tau$ and then follows our cusp geodesic from  the end point of $\tau$ to $w$ (see Figure \ref{F4}). Hence $f(\hat\tau)$ has image in the horoball over $gP_i$. 
 
 \medskip
 
Note that it may be case that $v_1=\bar v=\ast$ and $D_1=0$. Recall that $d(s_v'(n), s_v(n)\leq 3\delta$ for all $n$ (and similarly for $s_w'$ and $s_w$). If $t\in (0,1)$ and $(t,n)$ is a vertex of $F$ then $d(f(t,n),\ast)=n$. Hence each vertex of $[0,1]\times [D_1+L+K_1+8\delta,\infty)$ is mapped by $f$ to $X-B_{D_1+L+K_1+5\delta} (\ast)$. By Lemma \ref{filter}(2), $f$ maps the 1-skeleton of $[0,1]\times [D_1+L+K_1+8\delta,\infty)$ to $ X-B_{D_1+K_1+L+2\delta}(\ast)$. By Lemma \ref{filter}(3) the boundary of each open rectangle in $[0,1]\times [D_1+L+K_1+8\delta,\infty)-F$ is mapped by $f$ to $B_{7\delta+1}(f(z))$ for $z$ the upper left vertex of the rectangle.  The extension of $f$ to this rectangle has image in $St_L(f(z))$ (by our choice of $L$ and the definition of $f$ in Lemma \ref{filter}). 
Since $f(z)\in X-B_{D_1+L+K_1+5\delta}(\ast) $, the image of this rectangle under $f$ has image in $X-St_{D_1+K_1+5\delta}(\ast)$. Hence:

$$f([0,1]\times [D_1+L+K_1+8\delta,\infty))\subset X-St_{D_1+K_1+5\delta}(\ast).$$

Let $\bar f$ be a simplicial approximation to $f$ that agrees with $f$ on $F$. Note that $\bar f$ can only differ from $f$ on the open rectangles $R$ of $[0,1]\times [0,\infty)-F$ and $\bar f|_{\bar R}$ is a simplicial approximation of $f|_{\bar R}:\bar R\to St_L(z)$ (for $z$ the upper left (filter) vertex of $R$) that agrees with $f$ on the boundary of $R$. In particular,  $\bar f(\bar R)\subset St_L(z)$ and
$$\bar f([0,1]\times [D_1+L+K_1+8\delta,\infty))\subset X-St_{D_1+K_1+5\delta}(\ast)\subset X-St_{K_1}(\ast).$$

By Lemma \ref{ProjH} there is a proper simplicial homotopy $H_v$ of $Qr_v$ to $r_v$ rel$\{v\}$ with image in the $19\delta+1$ neighborhood of $r_v$ (and similarly there is $H_w$ for $r_w$ and $Qr_w$). Since $v$ and $w$ avoid $B_{K_2}(\ast)$, $K_1> N+19\delta +1$, and $K_2>3K_1$, the homotopies $H_v$ and $H_w$ avoid $St_{K_1}(\ast)$.
Now combining the proper simplicial homotopies $H_v$,  $\bar f$ (from $s_v'$ to $s_w'$) and $H_w$ gives a proper simplicial homotopy of $(\beta_v',Qr_v)$ to $(\beta_w',Qr_w)$. We combine this homotopy with an arbitrary simplicial homotopy of $Q\beta_v'$ to $\beta_v'$ and of $Q\beta_w'$ to $\beta_w'$ to obtain a proper simplicial homotopy $H$ of $(Q\beta_v',Qr_v) (=Qs_v')$ to $(Q\beta_w',Qr_w) (=Qs_w')$. 

Apply Theorem \ref{excise} to $H$. Each disk pair $(E_j,\alpha_j)$ is mapped by $H$ into a horoball with $H(\alpha_j)$ mapped into a $G$ translate of one of the $P_i$ and $H(E_j)$ mapped into the (open) horoball over that translate of $P_i$.    
By {\bf (A)}, one of these $E_j$, call it $D$, contains the path $\hat \tau$ and $H(D)$ has $v$ and $w$ in its boundary.  The boundary of  $D$ is composed of two simple edge paths (separated by $\hat \tau$) 
and $H$ composed with either connects $v$ and $w$. The definition of (the domain of) $\hat \tau$ implies one of these paths (call it $\hat \psi$) is above $\hat\tau$ and has image in $[0,1]\times [D_1+L+K+1+8\delta,\infty)$ union the domains of the homotopies $H_v$ from $Qr_v$ to $r_v$ and $H_w$ from $Qr_w$ to $r_w$. Each of these last two homotopies avoid $B_{K_1}(\ast)$. We have shown that $\bar f([0,1]\times [D_1+L+K_1+8\delta,\infty)\subset X-B_{D_1+K_1+5\delta-2}(\ast)\subset X-B_{K_1}(\ast)$. Hence $H$ composed with $\hat \psi$ and everything above $\hat \psi$ avoids $B_{K_1}(\ast)$. We are only interested in $H_1$, the restriction of the homotopy $H$ to the part of its domain above $\hat \psi$. We reparametrize the domain of $H_1$ and alter $H_1$ on certain disk pairs to obtain a homotopy $\hat H$, so that $H(\hat \psi)=\psi=\hat H|_{[0,1]\times \{0\}}$ (as mentioned in the statement of our Lemma \ref{periEdge}). 

\medskip

\noindent {\bf Claim 1.} Suppose $hP_j$ is a peripheral coset, $\bar H$ is the horoball over $hP_j$ and $c$ is a closest vertex of $\bar H(\delta) (=\mathcal D^{-1} (\delta)\cap \bar H$) to $\ast$. Then
there are only finitely many disk pairs $(D,\alpha)$ for $[0,1]\times [0,\infty)$ and the homotopy $H$, such that $H(\alpha)\subset hP_j$ and $D$ contains a vertex  $z_1=(a,b)$ of $F$ such that $\mathcal D(f(z_1))\geq L+2\delta +1$. Furthermore, each such disk $D$ contains a vertex $w_1'$ of our filter such that $d(f(w_1'),c)\leq 2L+5\delta+3$.

\begin{proof} 
Suppose $(D,\alpha)$ is such a disk pair. 
Let $\alpha_{z_1}$ be a geodesic in $X$ from $\ast$ to $f(z_1)$ (as in Lemma \ref{tight}). Say $\alpha_{z_1}(t)=x$ is the first point of $\alpha_{z_1}$ in $\bar H(\delta)$. By Lemma \ref{tight}, $d(x, c)\leq 2\delta+1$. 
The segment $\alpha_{z_1}([t,t+L+2\delta+1])$ of $\alpha_{z_1}$ (immediately following $x$) is vertical. Let $z=\alpha_{z_1}(t+L+2\delta+1)$. Let $\beta_{z_1}$ be the vertical segment of $[0,1]\times [0,\infty)$ from $(a,0) $ to $z_1=(a,b)$ (of length $b$). By Lemma \ref{filter}(4),  $f(\beta_{z_1})$ and $\alpha_{z_1}$ must $(L+2\delta)$-track one another. Hence if $w_1=\beta_{z_1} (t+2L+\delta+1)$, then $d(f(w_1), z)\leq L+2\delta$, and $f(\beta_{z_1}([t+L+2\delta+1,b]))\subset \bar H^1$. (See Figure \ref{F5}.) Note that $w_1$ belongs to an edge of our triangulation of our filter and so within 1 unit of a vertex $w_1'$ of our filter such that $f(w_1')\in \bar H^1$. 

\begin{figure}
\hspace{-3in}
\vbox to 3in{\vspace {-1in} \hspace {1.7in}
\includegraphics[scale=1]{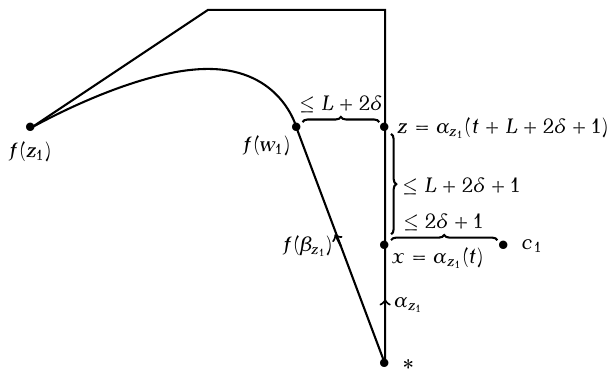}
\vss }
\caption{Tracking Paths in a Horoball} 
\label{F5}
\end{figure}

\medskip

In particular, $d(f(w_1), c)\leq d(f(w_1), z)+d(z, x)+d(x,c) \leq 2L+5\delta +2$ and $\beta_{z_1}([t+L+2\delta+1,b])\subset D$ so that $w_1,w_1'\in D$. In particular, $D$ contains a vertex $w_1'$, such that $f(w_1')$ is within $2L+5\delta+3$ of $c$. Since $f$ is proper, $[0,1]\times [0,\infty)$ contains only finitely many vertices that $f$ maps within  $2L+5\delta+3$ of $c$. Since the disks of the disk pairs are all disjoint, the claim follows.
\end{proof}

If $(D,\alpha)$ is a disk pair for $H$, arising from Theorem \ref{excise} and containing a vertex $v$ of the filter $F$ such that $\mathcal D(f(v))\geq L+2\delta+1$, then remove $D$ from $[0,1]\times [0,\infty)$. Recall that the image of the homotopy $H_1$ avoids $B_{K_1} (\ast)$, and if we can properly extend $H_1$ to the removed disks by a map that avoids  $B_N(\ast)$ and so that the extension has image in $X_{L+2\delta+2}$, we will have the desired homotopy $\hat H$ (with $\hat H|_{[0,1]\times \{0\}}=H_1(\hat \psi)=\psi$ after a reparametrization of the domain of $H_1$).  

\medskip

\noindent {\bf Claim 2.} If $(\alpha,D)$ is a disk pair of our triangulation of $[0,1]\times [0,\infty)$ and some vertex $y$ of $D$, is such that $\mathcal D(f(y))> 2L+2\delta+1$ then there is a filter vertex $z$ of $D$ with $\mathcal D(f(z))> L+2\delta +1$. 

\begin{proof} Since $\mathcal D(f(y)>2L+2\delta+1$,  Lemma \ref{Est} (with $K=3\delta$ and $E=2\delta +1$) implies that in our triangulation of $[0,1]\times [0,\infty)$ there is an edge path $\phi$ (of length $\leq L=A_{\ref{Trans} (7\delta+1}$) from $y$ to a filter vertex $z$ of $[0,1]\times [0,\infty)$ such that $H(\phi)$ is in the horoball containing $f(y)$ (and so $z$ is a vertex of $D$) and such that $\mathcal D(f(z))>L+2\delta+1$. 
\end{proof}

If $(D,\alpha)$ is a disk pair of our triangulation of $[0,1]\times [0,\infty)$ and $D$ does not contain a vertex of $v$ of $F$ with $\mathcal D(f(v))\geq L+2\delta+1$, then Claim 2 implies that $H(D)\subset X_{2L+2\delta +1}$.
Suppose $(D,\alpha)$ is a disk pair and $D$ contains a vertex $v$ of $F$ such that $\mathcal D(f(v))\geq L+2\delta +1$. Then the disk $D$ is removed from the domain of $H_1$ (the part of $[0,1]\times [0,\infty)$ above $\hat \psi$). Say $\alpha$ has image in $hP_j$. Let $\bar H$ be the horoball above $hP_j$ and let $c$ be a closest vertex of $\bar H(\delta) (=\mathcal D^{-1} (\delta)\cap \bar H$) to $\ast$. 
Claim 1 implies that $D$ contains a vertex $w$ from the filter $F$ such that $d(f(w),c)\leq 2L+5\delta+3$. Since $H_1$ avoids $B_{K_1}(\ast)$, and $B_{2L+5\delta+3}(c_i)\subset B_{K_1}(\ast)$ for $i\in \{1,\ldots, m\}$ our peripheral $hP_j$ cannot be in $\mathcal A=\{A_1,\ldots, A_m\}$. This implies that $hP_j$ avoids $B_N(\ast)$. Hence if $D$ is bounded, then any homotopy killing $\alpha$ in $\Gamma(hP_j)$ avoids $B_N(\ast)$. If $D$ is unbounded (and $\alpha$ is a line) then any proper homotopy in $\Gamma(hP_j)$ of two opposing rays of this line avoids $B_N(\ast)$ (such a homotopy exists since $P_j$ is 1-ended and has semistable fundamental group at $\infty$). Define $\hat H$ on $D$ to be such a homotopy. It suffices to show the resulting homotopy is proper. Given any compact set $C\subset X$ only finitely many peripheral subgroups intersect $C$. Hence only finitely many of the extensions of $H_1$ intersect $C$ so that $\hat H^{-1} (C)$ is contained in the compact set $H^{-1} (C)$ union the inverse image of finitely many extensions of $H_1$ to (finitely many) disks. Since each such extension (on the closed disk) is proper $\hat H$ is a proper map with image in $X_{2L+2\delta+1}(=X_{M_0})$. This concludes the proof of Case 1. 

\medskip

Before we consider the second case, we prove part (1) the lemma. We build the homotopy $\hat H$ in a similar way, but less care is necessary. The paths $\beta_v'$ and $\beta_w'$ are not necessary. Instead use the paths $\beta_v$  and $\beta_w$ (geodesic edge paths in $X$ from $\ast$ to $v$ and $w$ respectively)  to define $s_v'=(\beta_v,r_v)$ and $s_w'=(\beta_w,r_v)$. Build a filter homotopy between $Qs_v=(Q\beta_v,Qr_v)$ and $Qs_w$ and use relative simplicial approximation to obtain a proper simplicial homotopy $H$ between $Qs_v$ and $Qs_w$ with image in $X$. (See Figure \ref{F4}.) 
Next use Theorem \ref{excise}, Claim 1 and Claim 2 (as before) to obtain a proper homotopy $H_1$ of $Qs_v=(Q\beta_v, Qr_v)$ to $Qs_w=(Q\beta_w,Qr_w)$ rel$\{\ast\}$, with image in $X_{M_0}$. 

Since the loop $(Q\beta_v^{-1}, Q\beta_w,e^{-1})$ is homotopically trivial in $Y$, we can combine $H_1$ with such a homotopy and replace $(Q\beta_v^{-1}, Q\beta_w)$ by $e$.  We obtain a proper homotopy $\hat H$ of $Qr_v$ to $(e,Qr_w)$ with image in $X_{M_0}$. This finishes part (1) of the lemma.

\medskip

\noindent For the final case, we follow much of our earlier notation. Assume that $K'\geq N+19\delta+1$ and for $j\in \{1,\ldots, m\}$
$$B_{2L+5\delta +3}(c_j)\subset B_{K'}(\ast).$$
Let $t_1=K'+L+\delta$ and $M'=M_{\ref{LC}} (t_1)$. 

\medskip
 
\noindent {\bf Case 2.} Assume $gP_i\not\in \{A_1,\ldots, A_m\}$ and $e$ is an edge of $gP_i$ in $X-B_{M'}(\ast)$.

\medskip 

\noindent Say our edge path ray $r_v$ (at $v$) converges to $x\in \partial X$ and $r_w$ converges to $y\in \partial X$. Let $\alpha_v$ be a geodesic from $\ast$ to $v$.
By Lemma \ref{track}, the vertex  $s_v(n)$ is within $\delta$ of the $n^{th}$ vertex of $(\alpha_v, r_v$). Similarly for $s_w$. The vertex of $s_v$ that is $d(v,\ast)$ from $\ast$ and the vertex of $s_w$ that is $d(w,\ast)$ from $\ast$ are within $2\delta+1$ of one another. Since $M'<d(v,\ast)$, we have $d(s_v(M'),s_w(M'))\leq 2\delta+1$.  
By Lemma \ref{LC}, there is a path $\gamma$ in $\partial X$ from $x$ to $y$ such that for any two points $w_1$ and $w_2$ in the image of $\gamma$ and any geodesic edge paths $q_1$ and $q_2$ at $\ast $ converging to $w_1$ and $w_2$ respectively, $d(q_1(k), q_2(k))\leq \delta$ for all $k\leq t_1(=K'+L+\delta)$. 

Let $f:[0,1]\times [0,\infty)\to X$ be a filter homotopy for a filter $F(s_v,s_w,\gamma, 0)$ of Lemma \ref{filter}. Recall that on each rectangle $R$ of $[0,1]\times [0,\infty)-F$ we have $f(\bar R)\subset St_L(z)$ where $z$ is the upper left vertex of $R$. 
Let $f_1$ be the restriction of $f$ to $[0,1]\times [K'+L,\infty)$ and let $\tau$ be $f_1$ restricted to $[0,1]\times \{t_1\}$. Let $F_1$ be a proper simplicial approximation to $f_1$ that agrees with $f_1$ on the filter $F$ and with image in the $St_L$ neighborhood of the part of the filter in $[0,1]\times [t_1,\infty)$. (The map $F_1$ is obtained by combining simplicial approximations to $f_1$ on closed rectangles.) Since $f_1$ restricted to the part of the filter in $[0,1]\times [t_1,\infty)$ avoids $B_{K'+L}(\ast) $, the image of $F_1$ is in $X-B_{K'}(\ast)$. 
By Lemma \ref{LC}, the path $\tau$ has image in $N_\delta(s_v(t_1))$. 
Next we define a proper simplicial homotopy $H$ of $Q(r_v)$ to $(e,Q(r_w)$, rel$\{v\}$ in $X-B_N(\ast)$ by combining $F_1$ with six other proper simplicial homotopies (see  Figure \ref{F6}).

\begin{figure}
\hspace{-3in}
\vbox to 3in{\vspace {-1in} \hspace {1in}
\includegraphics[scale=1]{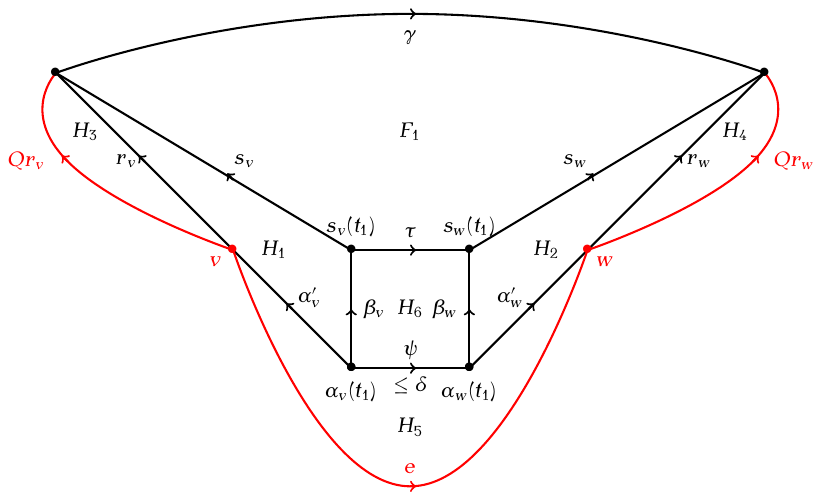}
\vss }
\vspace{.8in}
\caption{Combining Homotopies} 
\label{F6}
\end{figure}

Let $\alpha_v'$ be the tail of $\alpha_v$ beginning at $\alpha_v(t_1)$. Let $\beta _v$ be an edge path of length $\leq \delta$  from $\alpha_v(t_1)$ to $s_v(t_1)$. 
For $m\geq t_1$, consider the edges $\alpha_v ([m,m+1])$ and $s_v([m,m+1])$, and paths of length $\leq \delta$ from $\alpha_v(m)$ to $s_v(m)$ and from $\alpha_v(m+1)$ to $s_v(m+1)$, forming loops. These loops are homotopically trivial in $St_{A_{\ref{Trans}(\delta+1)}}(s_v(m))\subset St_L(s_v(m))$ and so by Theorem \ref{Z} we may assume these homotopies are simplicial. Combining these homotopies,
there is a proper simplicial homotopy $H_1$ between $(\beta _v,s_v|_{[t_1,\infty)})$ and $(\alpha_v',r_v)$ with image in $St_{L}(s_v([t_1,\infty)))\subset X-B_{K'}(\ast)$.  Similarly for $H_2$ and $w$. 

By Lemma \ref{ProjH} there is a simplicial homotopy $H_3$ of $r_v$ to a projection $Qr_v$, rel$\{v\}$ with image in the $19\delta+1$ neighborhood of $im(r_v)$. Similarly for $H_4$ and $r_w$. Consider the geodesic triangle formed by $\alpha_v$, $\alpha_w$ and $e=(v,w)$. The $m^{th} $ vertex of $\alpha_v$ is within $\delta$ of the $m^{th}$ vertex of $\alpha_w$ for all $m$.  Let $\psi$ be an edge path of length $\leq \delta$ from $\alpha_v(t_1)$ to $\alpha_w(t_1)$. (For simplicity, assume that $d(\ast,v)=d(\ast,w)$.) For each integer $t_1\leq   m \leq d(\ast,v)$ there is an edge path $\psi_m$ from $\alpha_v(m)$ to $\alpha_w(m)$ (with $\psi_{t_1}=\psi$  and $\psi_{d(\ast,v)}=e)$. The loops formed by $\psi_m$, $\psi_{m+1}$ and the corresponding edge from $\alpha_v'$ and $\alpha_w'$ is homotopically trivial by a simplicial homotopy in $St_{\ref{Trans}(\delta+1)}(\alpha(m))$. Combining these homotopies gives the homotopy $H_5$, a simplicial homotopy of $(\alpha_v',e)$ to $(\psi,\alpha_w')$ in the $A_{\ref{Trans}(\delta+1)}$ star neighborhood of $\alpha_v'$. Hence $H_5$ has image in $X-St_{K'}(\ast)$. The loop $(\beta_v, \tau, \beta_w^{-1},\psi^{-1})$ has image in $B_{2\delta}(s_v( t_1))$ and so is homotopically trivial (by the simplicial homotopy $H_6$) in $St_{L}(s_v(t_1))\subset X-St_{K'}(\ast)$.

Combining these homotopies, we have $H$, a proper simplicial homotopy rel$\{v\}$ of $Qr_v$ to $(e,Qr_w)$ (see Figure \ref{F6}), with image in $X-St_{K'}(\ast)$. Now, use Theorem \ref{excise} (to cut out the disks of $[0,1]\times [0,\infty)$ that $H$ does not map into $X_{M_0}$). Define $\hat H$ to agree with $H$ on the compliment of the removed disks. Suppose $(D, \alpha)$ is such a disk pair and $H(\alpha)$ has image in the  peripheral $hP_j$. Let $\tilde H$ be the horoball over $hP_j$ in $X$ and let $c$ be a closest point of $\tilde H(\delta)$ to $\ast$. Claims 1 and 2 imply $D$ contains a filter vertex within $2L+5\delta+3$ of $c$. Since $B_{2L+5\delta+3}(c_i)\subset B_{K'}(\ast)$ for $i\in \{1,\ldots, m\}$, we have $hP_j\not\in \{A_1,\ldots, A_m\}$.

We proceed just as before. If $D$ is bounded, then any homotopy killing $\alpha$ in $\Gamma(hP_j)$ avoids $B_N(\ast)$. If $D$ is unbounded (and $\alpha$ is a line) then any proper homotopy in $\Gamma(hP_j)$ of two opposing rays of this line avoids $B_N(\ast)$.  Define $\hat H$ on $D$ to be such a homotopy. Just as before, $\hat H$ is proper with image in $X_{M_0}- B_N(\ast)$. This completes the proof of Case 2 and Lemma \ref{periEdge}.
\end{proof}

The proof of our Main Theorem will be derived from the next result by a homotopy ``stacking" argument.

\begin{lemma} \label{MLe} 
If $d=(w,q)$ is an edge of $Y$ then $Qr_w$ is properly homotopic rel$\{w\}$ to $(d,Qr_q)$ in $X_{M_0}$. Furthermore, for each integer $N$ there is an integer $M_{\ref{MLe}}(N)$ such that if $d$ has image in $Y-B_{M_{\ref{MLe}}(N)}$ then $Qr_w$ is properly homotopic rel$\{w\}$ to $(d,Qr_q)$ by a homotopy in $X_{M_0}-B_N(\ast)$. 
\end{lemma}

\begin{proof}
If $d$ belongs to a peripheral coset then $Qr_w$ is properly homotopic rel$\{w\}$ to $(d,Qr_q)$ in $X_{M_0}$ by Lemma \ref{periEdge}. Otherwise, let $e=(v,w)$ be an edge of a peripheral coset. By Lemma \ref{periEdge}, the ray $Qr_v$ is properly homotopic rel$\{v\}$ to $(e, Qr_w)$ in $X_{M_0}$. Equivalently, $Qr_w$ is properly homotopic rel$\{w\}$ to $(e^{-1}, Qr_v)$ in $X_{M_0}$. 
Again by Lemma \ref{periEdge} $Qr_v$ is properly homotopic rel$\{v\}$ to $(e, d,Qr_q)$ in $X_{M_0}$. Equivalently, $(e^{-1},Qr_v)$ is properly homotopic rel$\{w\}$ to $(d,Qr_q)$ in $X_{M_0}$.  Since both $(d,Qr_q)$ and $Qr_w$ are properly homotopic rel$\{w\}$ to $(e^{-1},Qr_v)$ in $X_{M_0}$, the first part of the lemma is proved.

For the second part, we will show that $Qr_w$ and $(d,Qr_q)$ are properly homotopic rel $\{w\}$ to `far out' proper rays at $w$ that have image in a peripheral coset. Since peripheral subgroups are semistable, these rays are in turn properly homotopic rel$\{w\}$ to one another in $X_{M_0}-B_N(\ast)$. Combining homotopies will finish the proof of the lemma.

Without loss, assume the integers $M_{\ref{periEdge}}(k)$ are strictly increasing in $k$. We choose $P_1\in$ {\bf P} (any other peripheral would do as well).  Let $\mathcal A_N$ be the (finite) set of all peripheral cosets $vP_1$ such that $v\in B_N(\ast)$. 
Since $P_1$ has semistable fundamental group at $\infty$, Theorem \ref{ssequiv}(3) implies there is an integer $M_1(N)$ 
such that if $A\in \mathcal A_N$ and $r$ and $s$ are proper edge path rays in $A-St_{M_1(N)}(\ast)$, both based at the vertex $v$, then $r$ and $s$ are properly homotopic rel$\{v\}$ in $\Gamma(A)-St_N(\ast)$. Let $M_2(N)=M_{\ref{periEdge}}(M_1(N))$.
Let $\mathcal B_N$ be the (finite) set of all peripherals $vP_1$ such that $v\in B_{M_2(N)}(\ast)$. 
Let $M_3(N)$ be such that if $B\in \mathcal B_N$ then the bounded components of  $\Gamma (B)- St_{M_2(N)}(\ast)$ belong to $B_{M_3(N)}(\ast)$. We will show that $M_3(N)$ satisfies the role of $M_{\ref{MLe}}(N)$. 

Let $d=(w,q)$ be an edge in $Y-B_{M_3(N)}(\ast)$ and let $A$ be the peripheral coset $wP_1$. The constant $M_3(N)$ has been chosen so that whether or not $A\in \mathcal B_N$, there is an proper edge path ray $r=(e_1,e_2,\ldots)$ based at $w$ and with image in $\Gamma(A)-B_{M_2(N)}(\ast)$.
Label the consecutive vertices of $r$ as $v_0=w,v_1,v_2,\ldots$. 

For $k\geq 1$ let $N_k$ be the largest integer such that $e_k$ is in $Y-B_{M_{\ref{periEdge}}(N_k)}(\ast)$. By the definition of $M_2(N)$, we have $N_k\geq M_1(N)$ for all $k$. By Lemma \ref{periEdge} there is a proper homotopy  rel$\{v_{i-1}\}$ (call it $H_i$) of $Qr_{v_{i-1}}$ to $(\psi_i,Qr_{v_{i}})$ in $X_{M_0}-B_{N_k}(\ast)$ where $\psi_i$ is an edge path in $A-B_{N_k}(\ast)\subset A-B_{M_1(N)}(\ast)$ from $v_{i-1}$ to $v_{i}$. 
$$H_i:Qr_{v_{i-1}}\sim _{v_{i-1}}(\psi_i,Qr_{v_i}).$$
Since $r$ is proper, the $e_k$ converge to infinity and so the $N_k$ converge to infinity. This means that the images of only finitely many $H_i$ intersect any given compact set. Hence combining the $H_i$ as in Figure \ref{F7} gives $\hat H_1$, a proper homotopy rel$\{w\}$ of $Qr_{v_0}=Qr_w$ to $r_1=(\psi_1,\psi_2,\ldots)$ with image in $X_{M_0}-B_{M_1(N)}(\ast)$. 
$$\hat H_1:Qr_{v_0}=Qr_w\sim_w r_1=(\psi_1,\psi_2,\ldots)\to X_{M_0}-B_{M_1(N)}(\ast).$$
Lemma \ref{periEdge} also gives a proper homotopy rel$\{v_1\}$ (call it $\hat H_2$) of $Qr_{v_1}$ to $(\phi,d,Qr_q)$ in $X_{M_0}-B_{M_1(N)}(\ast)$, where $\phi$ is an edge path in $A-B_{M_1(N)}(\ast)$ from $v_1$ to $v_0$.
Equivalently $\hat H_2$ is a proper homotopy rel$\{w\}$ of $(d,Qr_{q})$ to $(\phi^{-1}, Qr_{v_1})$ in $X_{M_0}-B_{M_1(N)}(\ast)$. 
$$\hat H_2:(d,Qr_{q})\sim_{w}(\phi^{-1},Qr_{v_1})\to X_{M_0}-B_{M_1(N)}(\ast).$$

\begin{figure}
\hspace{-3in}
\vbox to 3in{\vspace {-1in} \hspace {1.5in}
\includegraphics[scale=1]{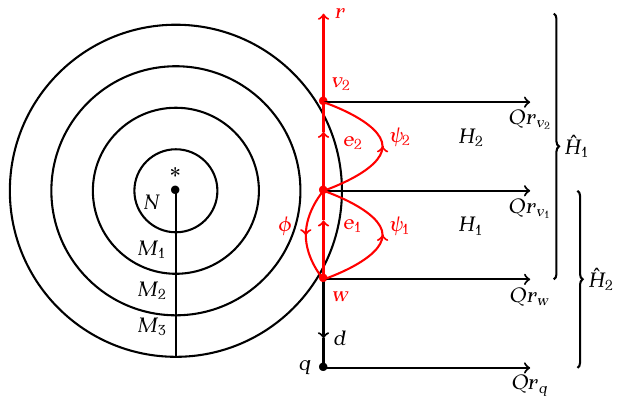}
\vss }
\vspace{.3in}
\caption{Multiple Homotopies} 
\label{F7}
\end{figure}

Combining the homotopies $H_2,H_3,\ldots$ gives $\hat H_3$, a proper homotopy $\hat H_3$ rel$\{v_1\}$ of $Qr_{v_1}$ to $(\psi_2,\psi_3,\ldots)$ in $X_{M_0}-B_{M_1(N)}(\ast)$.
$$\hat H_3: Qr_{v_1}\sim_{v_1} (\psi_2,\psi_3,\ldots)\to X_{M_0}-B_{M_1(N)}(\ast).$$
Combining $\hat H_2$ and $\hat H_3$ gives $\hat H_4$ a proper homotopy rel$\{w\}$ of $(d,Qr_{q})$ to $(\phi^{-1}, \psi_2,\psi_3,\ldots)$ in $X_{M_0}-B_{M_1(N)}(\ast)$.
$$\hat H_4: (d,Qr_{q})\sim_{w} (\phi^{-1},\psi_2,\psi_3,\ldots)\to X_{M_0}-B_{M_1(N)}(\ast).$$
Whether or not $A\in \mathcal A_N$, the definition of $M_1(N)$, implies there is a proper homotopy rel$\{w\}$ (call it $\hat H_5$) of $(\phi^{-1},\psi_2,\psi_3,\ldots)$ to $(\psi_1,\psi_2, \ldots)$ in $\Gamma(A)-St_N(\ast)$.  
$$\hat H_5:(\phi^{-1},\psi_2,\psi_3,\ldots)\sim_w(\psi_1,\psi_2, \ldots)\to \Gamma(A)-St_N(\ast).$$
Combining $\hat H_1$, $\hat H_5$ and $\hat H_4$ gives a proper homotopy rel$\{w\}$ of  $Qr_w$ to $(d,Qr_q)$ by a homotopy in $X_{M_0} -B_N(\ast).$
$$Qr_w\sim_w(\psi_1,\psi_2,\ldots) \sim_w(\phi^{-1},\psi_2,\psi_3,\ldots) \sim_{w}(d,Qr_{q})$$
\end{proof}

\begin{theorem}\label{Main1} 
Suppose $G$ is a 1-ended finitely presented group that is hyperbolic relative to {\bf P} a finite collection of 1-ended finitely presented proper subgroups of $G$. If each $P\in {\bf P}$ has semistable fundamental group at $\infty$, then $G$ has  semistable fundamental group at $\infty$. 
\end{theorem}
\begin{proof} 
If $r$ is a proper edge path ray in $X_{M_0}$ and based at $\ast$, then $r$ is properly homotopic to any projection $Qr$ of $r$ to $Y$. Hence we need only consider proper edge path rays based at $\ast$ and with image in $Y$.   Let $r_{\ast}=l^+$. We show for any proper edge path ray $s$ at $\ast$ and with image in $Y$, $s$ is properly homotopic rel$\{\ast\}$ to $Qr_\ast$ in $X_{M_0}$. Then, if $s_1$ and $s_2$ are arbitrary proper edge path rays at $\ast$ and with image in $Y$ we have both are properly homotopic rel$\{\ast\}$ to $r_\ast$ in $X_{M_0}$ and hence $s_1$ is properly homotopic to $s_2$ rel$\{\ast\}$ in $X_{M_0}$. This means $X_{M_0}$ has semistable fundamental group at $\infty$. Equivalently, $G$ has semistable fundamental group at $\infty$.

Write $s$ as the edge path $(e_1,e_2,\dots)$ and say that $v_i$ is the initial vertex of $e_i$. Let  $0<N_1<N_2<\cdots$ be a sequence of integers such that $M_{\ref{MLe} }(N_i)< N_{i+1}$ for all $i\geq 1$. 
Since $s$ is proper, there is an integer $K_2$ such that for all $i\geq K_2$, $e_i$ has image in $Y-B_{N_2}(\ast)$. Given an integer $j>2$ there is an integer $K_j\geq K_{j-1}$ such that for all $i\geq K_j$, $e_i$ has image in $Y-B_{N_j}(\ast)$.  For $1\leq i<K_2$, Lemma \ref{MLe} implies there is a proper homotopy $H_i$ rel$\{v_i\}$ of $Qr_{v_i}$ to $(e_i, Qr_{v_{i+1}})$.

For $j\geq 2$ and $K_j\leq i< K_{j+1}$, the edge $e_i$ has image in $Y-B_{N_j}(\ast)$. For such $i$, we use Lemma \ref{MLe} to obtain a proper homotopy $H_i$ rel$\{v_i\}$ of $Qr_{v_i}$ to $(e_i, Qr_{v_{i+1}})$ with image in $X_{M_0}-B_{N_{j-1}}(\ast)$. Let $H$ be the homotopy obtained by combining the homotopies $H_i$ as in Figure \ref{F8}. 

\begin{figure}
\hspace{-3in}
\vbox to 3in{\vspace {-1in} \hspace {2.6in}
\includegraphics[scale=1]{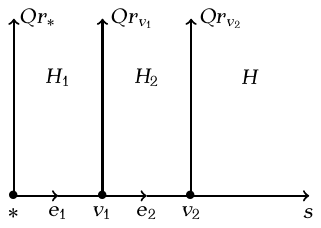}
\vss }
\vspace{-.9in}
\caption{Final Homotopies} 
\label{F8}
\end{figure}

For $j\geq 2$ and $K_j\leq i< K_{j+1}$, the edge $e_i$ has image in $Y-B_{N_j}(\ast)$. For such $i$, we use Lemma \ref{MLe} to obtain a proper homotopy $H_i$ rel$\{v_i\}$ of $Qr_{v_i}$ to $(e_i, Qr_{v_{i+1}})$ with image in $X_{M_0}-B_{N_{j-1}}(\ast)$. Let $H$ be the homotopy obtained by combining the homotopies $H_i$ as in Figure \ref{F8}.

It suffices to show that $H$ is proper. Let $C$ be compact in $X_{M_0}$ and $j$ such that $C\subset B_{N_j}(\ast)$. Then for all $k\geq K_{j+1}$, $H_k$ has image in $X_{M_0}-B_{N_j}(\ast)\subset X_{M_0}-C$. Then $H^{-1}(C)=\cup_{i=1}^{K_{j+1}} H_i^{-1}(C)$ is a finite union of compact sets and $H$ is proper. 
\end{proof}

\bibliographystyle{amsalpha}
\bibliography{paper}{}

\end{document}